\documentclass[10pt,reqno,twoside]{amsart}
\usepackage{amsmath,amssymb,amsfonts,mathrsfs,fancyhdr,colortbl,color}
\usepackage[hypertex]{hyperref}\hypersetup{colorlinks=black,citecolor=blue}
\usepackage[numbers,sort&compress]{natbib}
\theoremstyle{plain}

\newtheorem{thm}{Theorem}[section]

\newtheorem{cor}[thm]{Corollary}

\theoremstyle{Definition} \theoremstyle{remark}

\newtheorem{rem}[thm]{\rm \bf Remark}

\newcommand{\thmref}[1]{Theorem~\ref{#1}}

\newcommand{\abs}[1]{\left\vert#1\right\vert}

 \numberwithin{equation}{section}

\begin{document}
\thispagestyle{empty}
\begin{center}
{\large \bf Some extensions for Ramanujan's circular summation formula}
\end{center}
\vspace{0.2cm}
\begin{center}
{\large \bf Ji-Ke Ge$^1$ and Qiu-Ming Luo$^{2,\ast}$}\\
\vspace{0.2cm}
$^1$School of Electrical and Information Engineering\\Chongqing University of Science and Technology\\Chongqing Higher Education Mega Center, 
Huxi Campus\\Chongqing 401331, People's Republic of China\\
{\large \bf \hspace{-0.5cm} E-Mail: gjkweb@126.com}\\
\vspace{0.1cm}
$^2$Department of Mathematics, Chongqing Normal University\\Chongqing Higher Education Mega Center, Huxi Campus\\Chongqing 401331, People's Republic of China\\
{\large \bf E-Mail: luomath2007@163.com}\\
\vspace{0.2cm} {\bf $^{\ast}$Corresponding author}
\end{center}
\vspace{0.3cm}
\begin{center}
\textbf{Abstract}
\end{center}
\begin{quotation}
In this paper, we give some extensions for Ramanujan's circular summation formula with the mixed products of two Jacobi's theta functions. As some applications, we also obtain many interesting identities of Jacobi's theta functions.
\end{quotation}

\noindent
\textbf{2010 \textit{\textbf{Mathematics Subject Classification}}. } Primary 11F27; Secondary 11F20, 33E05.

\noindent
\textit{\textbf{Key Words and Phrases}}\textbf{. } Elliptic functions; Jacobi's theta functions; Ramanujan's circular summation; identities of Jacobi's theta functions.
\bigskip

\section{\bf Introduction and main results}
Throughout this paper we take $q=e^{\pi{i}\tau}$, $\Im (\tau)>0$, $z \in \mathbb{C}$.

The classical Jacobi's theta functions $\vartheta_i\left(z|\tau\right), i=1,2,3,4,$ are defined by
\begin{align}\label{Theta1}
\vartheta_1\left(z|\tau\right)&=-iq^{1/4}\sum_{n=-\infty}^\infty\left(-1\right)^{n}q^{n\left(n+1\right)}e^{\left(2n+1\right)iz}, \\
\label{Theta2}
\vartheta_2\left(z|\tau\right)&=q^{1/4}\sum_{n=-\infty}^{\infty}
q^{n\left(n+1\right)}e^{\left(2n+1\right)iz}, \\
 \label{Theta3}
\vartheta_3\left(z|\tau\right)&=\sum_{n=-\infty}^\infty
q^{n^{2}}e^{2niz}, \\
 \label{Theta4}
\vartheta_4\left(z|\tau\right)&=\sum_{n=-\infty}^\infty
\left(-1\right)^{n}q^{n^{2}}e^{2niz}.
\end{align}
We have
\begin{align}\label{Theta5}
& \vartheta_1\left(z+\pi\left|\tau\right.\right)=-\vartheta_1\left(z|\tau\right), \quad  \vartheta_1\left(z+\pi\tau\left|\tau\right.\right)=-q^{-1}e^{-2iz}\vartheta_1\left(z|\tau\right),\\
\label{Theta6}
& \vartheta_2\left(z+\pi\left|\tau\right.\right)=-\vartheta_2\left(z|\tau\right),\quad  \vartheta_2\left(z+\pi\tau\left|\tau\right.\right)=q^{-1}e^{-2iz}\vartheta_2\left(z|\tau\right), \\
\label{Theta7}
& \vartheta_3\left(z+\pi\left|\tau\right.\right)=\vartheta_3\left(z|\tau\right),\quad  \vartheta_3\left(z+\pi\tau\left|\tau\right.\right)=q^{-1}e^{-2iz}\vartheta_3\left(z|\tau\right),\\
\label{Theta8} &
\vartheta_4\left(z+\pi\left|\tau\right.\right)=\vartheta_4\left(z|\tau\right),\quad
\vartheta_4\left(z+\pi\tau\left|\tau\right.\right)=-q^{-1}e^{-2iz}\vartheta_4\left(z|\tau\right).
\end{align}
By applying the induction, we easily obtain
\begin{align}\label{Theta9}
& \vartheta_1\left(z+n\pi\tau\left|\tau\right.\right)=(-1)^nq^{-n^2}e^{-2niz}\vartheta_1\left(z|\tau\right),\\
\label{Theta10}
& \vartheta_2\left(z+n\pi\tau\left|\tau\right.\right)=q^{-n^2}e^{-2niz}\vartheta_2\left(z|\tau\right),\\
\label{Theta11}
& \vartheta_3\left(z+n\pi\tau\left|\tau\right.\right)=q^{-n^2}e^{-2niz}\vartheta_3\left(z|\tau\right),\\
\label{Theta12} &
\vartheta_4\left(z+n\pi\tau\left|\tau\right.\right)=(-1)^nq^{-n^2}e^{-2niz}\vartheta_4\left(z|\tau\right).
\end{align}

On page 54 in Ramanujan's lost notebook (see \cite[p. 54, Entry 9.1.1 ]{Rama1988}, or \cite[p. 337]{AndrewsB2012}), Ramanujan recorded the following claim (without proof) which is now well known as Ramanujan's circular summation. The appellation circular summation was initiated by Son (see \cite[p. 338]{AndrewsB2012}).
\begin{thm}[Ramanujan's circular summation]\label{Luo116}
For each positive integer $n$ and $\abs{ab} < 1$, 
\begin{equation}\label{Luo117}
\sum_{-n/2<r \leq n/2} \left(\sum_{\substack{k=-\infty \\ k \equiv r (\textup{mod} \ n)}}^{\infty}   a^{{k(k+1)}/{(2n)}}b^{{k(k-1)}/{(2n)}}\right)^n=f(a,b)F_n(ab),
\end{equation}
where
\begin{equation}\label{Luo138}
F_n(q):=1+2nq^{(n-1)/2}+\cdots, \qquad n \geq 3.
\end{equation}
Ramanujan's theta function $f(a,b)$ is defined by
\begin{equation}\label{Luo110}
f(a,b)=\sum_{n=-\infty}^{\infty}a^{{n(n+1)}/{2}}b^{{n(n-1)}/{2}}, \quad \abs{ab}<1.
\end{equation}
\end{thm}

By the definition of Ramanujan's theta function above and routine calculations, we can rewrite Ramanujan's circular summation \eqref{Luo117} as follows (see, for details \cite[p. 338]{AndrewsB2012}).
\begin{thm}[Ramanujan's circular summation]\label{Luo120}
Let $U_k=a^{{k(k+1)}/{(2n)}}$ and $V_k=b^{{k(k-1)}/{(2n)}}$. For each positive integer $n$ and $\abs{ab} < 1$, then 
\begin{equation}\label{Luo119}
\sum_{k=0}^{n-1} U_k^n f^n \left(\frac{U_{n+k}}{U_k}, \frac{V_{n-k}}{V_k}\right)=f(a,b)F_n(ab),
\end{equation}
where
\begin{equation}\label{Luo118}
F_n(q):=1+2nq^{(n-1)/2}+\cdots, \qquad n \geq 3.
\end{equation}
\end{thm}
When $n = 1$, the identity \eqref{Luo117} of \thmref{Luo116} merely reduces to the tautology $f(a, b) = f(a, b)$. When $n = 2$, the identity \eqref{Luo117} of \thmref{Luo116} holds if the coefficient 2 in \eqref{Luo138} is deleted. 

The claim has been first proven by Rangachari \cite{Rang1994} who also verified Ramanujan's explicit and elegant formulae for $F_n(x)$ for $n = 2, 3, 4, 5, 7$ by employing Mumford's theory of theta functions and root lattices.  Several authors have determined the identification of $F_n(x)$ in further special cases. S. Ahlgren \cite{Ahlgren2000} considered the cases $n = 6, 8, 9$, and K. Ono \cite{Ono1999} established $F_{11}(x)$, while K.S. Chua \cite{Chua2001} derived the corresponding result for $F_{13}(x)$. K.S. Chua \cite{Chua2002} and T. Murayama \cite{Murayama1998}, independently, improved the work of Rangachari by removing a condition of primality from Rangachari's work. S.H. Son \cite{Son2004} devised a proof of \eqref{Luo117} that is more in tune with Ramanujan's work, Son used functional equations in the spirit of $q$-series. A summary of all known dentifications
of $F_n(x)$ can be found in Son's paper \cite{Son2004}.

If we are going to apply the transformation $\tau \longmapsto -\frac{1}{\tau}$ to Ramanujan's identity, it will be convenient to convert Ramanujan's theorem into one involving the classical theta function $\vartheta_3(z|\tau)$ defined by \eqref{Theta3}. H. H. Chan, Z.-G. Liu and S. T. Ng \cite{ChanHH2006a} prove that \thmref{Luo121} below is equivalent to \thmref{Luo116}.

\begin{thm}[Ramanujan's circular summation]\label{Luo121}
For any positive integer $n\geq 2$, we have
\begin{equation}\label{Luo122}
\sum_{k=0}^{n-1}q^{k^2}e^{2k iz}\vartheta^n_{3}(z+k\pi\tau|n\tau)=
\vartheta_3(z|\tau)F_n(\tau).
\end{equation}
When $n\geq 3$,
\begin{equation}
F_n(q)=1+2nq^{n-1}+\cdots.
\end{equation}
\end{thm}

H. H. Chan, Z.-G. Liu and S. T. Ng \cite{ChanHH2006a} also showed that \thmref{Luo123} below is equivalent to \thmref{Luo116}, \thmref{Luo120} and \thmref{Luo121} by applying Jacobi's imaginary transformation formula \cite[p. 475]{Whittaker}.

\begin{thm}[Ramanujan's circular summation]\label{Luo123}
For any positive integer $n$, we have
\begin{equation}\label{Luo124}
\sum_{k=0}^{n-1}\vartheta^n_{3}\left(z+\frac{k\pi}{n}\left|\tau\right.\right)=G_n(\tau)\vartheta_3(nz|n\tau),
\end{equation}
where
\begin{equation}\label{Luo125}
G_n(\tau)=n
 \sum_{\substack{r_{1}+\cdots+r_{n}=0\\ r_{1}, \ldots, r_{n}=-\infty}}^{\infty}q^{r^{2}_{1}+\cdots+r^{2}_{n}}=\sqrt{n} (-i \tau)^{(1-n)/2}F_n \left(-\frac{1}{n \tau}\right).
\end{equation}
\end{thm}
Many years ago, M. Boon \textit{et al.} \cite[p. 3440, Eq. (7)]{Boon1982} obtained the following additive decomposition of $\vartheta_3(z|\tau)$.
\begin{thm}\label{Luo126}
For any positive integer $n$, we have
\begin{equation}\label{Luo127}
\sum_{k=0}^{n-1}\vartheta_3(z+k\pi|\tau)= n\vartheta_3(nz|n^2\tau).
\end{equation}
\end{thm}
Recentely, Zeng \cite{Zeng2009} proved the following theorem which unifies \thmref{Luo123} and \thmref{Luo126}.
\begin{thm}\label{Luo128}
For any nonegative integer $k,n,a$ and $b$ with $a+b=n$,
\begin{align}\label{Luo129}
\sum_{s=0}^{kn-1}\vartheta_3^{a}\left(\frac{z}{kn}+\frac{y}{a}+\frac{\pi
s}{kn}\left|\frac{\tau}{kn^{2}}\right.\right)\vartheta_3^{b}\left(\frac{z}{kn}-\frac{y}{b}+\frac{\pi
s}{kn}\left|\frac{\tau}{kn^{2}}\right.\right)=R_{33}\left(a,b;\frac{y}{ab},\frac{\tau}{kn^2} \right)\vartheta_3\left(z|\tau\right),
\end{align}
where
\begin{align*}
R_{33}\left(a,b;y,\tau\right)=kn \sum_{\substack{m_{1}+\cdots+m_{a}+n_{1}+\cdots+n_{b}=0 \\ m_{1},\cdots,m_{a},n_{1},\cdots,n_{b}=-\infty}}^{\infty}
q^{m_{1}^{2}+\cdots+m_{a}^{2}+n_{1}^{2}+\cdots+n_{b}^{2}}e^{2k(m_{1}+\cdots+m_{a})iy}.
\end{align*}
\end{thm}

S. H. Chan and Z.-G. Liu \cite{ChanSH2010a} further extended Zeng's result as the following form.
\begin{thm} \label{Luo130}
Suppose $y_1,y_2,\cdots,y_n$ are $n$ complex numbers such that $y_1+y_2+ \cdots+ y_n=0$, 
\begin{equation} \label{Luo131}
\sum_{k=0}^{mn-1}\prod_{j=1}^{n} \vartheta_3\left(z+y_j+{k\pi\over mn}\left|\tau\right.\right)=
G_{m,n}(y_1,y_2,\cdots,y_n|\tau)\vartheta_3(mnz|m^2n \tau),
\end{equation}
where
\begin{equation}\label{Luo132}
G_{m,n}(y_1,y_2,\cdots,y_n|\tau) =mn \sum_{\substack{r_1,\cdots,r_n=-\infty\\r_1+\cdots+r_n=0}}^\infty
q^{r_1^2+r_2^2+\cdots+r_n^2}e^{2i(r_1y_1+r_2y_2+\cdots+r_ny_n)}.
\end{equation}
\end{thm}
Clearly, if $a$ and $b$ are two positive integers such that $a+b =n$, then we can take $y_1 = y_2 =\cdots= y_a = y/a$ and $y_{a+1} = y_{a+2} = \cdots = y_n =-y/b$ in \thmref{Luo130} to obtain \thmref{Luo128}.

Zhu \cite{Zhu2012a,Zhu2012b} study the alternating circular summation formula of theta functions and also correct an error of \cite{ChanSH2010a}. Cai \textit{et al.} \cite{luo2013a,luo2014a} obtain other circular summation formulas of theta functions. Z.-G. Liu \cite[p. 1978, Theorem 1.1 and Theorem 1.2]{LiuZG2012a} further obtained two theorems which unify \thmref{Luo130} and \cite[p. 117, Theorem 1.7]{Zhu2012a} of Zhu.

In the present paper, motivated by \cite{luo2013a},  \cite{ChanSH2010a} and \cite{ChanHH2006a}, we further give some extensions for Ramanujan's circular summation by applying the method of elliptic function. We also give some applications of the circular summation formulas. We now state our results as follows.

\begin{thm}\label{LUO4}
Let $n$ be even, $m$ be any positive integers, $a$ and $b$ be any non-negative integers such that $a+b=n$, and $x_{1}, \ldots,x_{a},y_{1},\ldots, y_{b}$ be any complex numbers such that their sum be $0$.
\begin{itemize}
\item When $ma$ is even, we have
\end{itemize}
\begin{align}\label{Luo40}
\sum_{k=0}^{mn-1}\prod_{j=1}^{a}\vartheta_1\left(z+x_{j}+\frac{k\pi
}{mn}\left|\tau\right.\right)\prod_{i=1}^{b}\vartheta_2\left(z+y_{i}+\frac{k\pi
}{mn}\left|\tau\right.\right)=R_{1,2}\left(m,n;\tau \right)\vartheta_3\left(mnz|{m^2n}\tau\right).
\end{align}
\begin{itemize}
\item When $ma$ is odd, we have
\end{itemize}
\begin{align}\label{Luo40a}
\sum_{k=0}^{mn-1}\prod_{j=1}^{a}\vartheta_1\left(z+x_{j}+\frac{k\pi
}{mn}\left|\tau\right.\right)\prod_{i=1}^{b}\vartheta_2\left(z+y_{i}+\frac{k\pi
}{mn}\left|\tau\right.\right)=R_{1,2}\left(m,n;\tau \right)\vartheta_4\left(mnz|{m^2n}\tau\right),
\end{align}
where
\begin{align}\label{Luo41}
R_{1,2}\left(m,n;\tau \right)
=& mni^{a}q^{-\frac{n}{4}}
\sum_{\substack{r_{1},\ldots,r_{a},s_{1},\ldots,s_{b}=-\infty
\\
2(r_{1}+\cdots+r_{a}+s_{1}+\cdots+s_{b})+n=0}}^{\infty}\left(-1\right)^{r_{1}+\cdots+r_{a}}\notag \\
& \times
q^{r_{1}^{2}+\cdots+r_{a}^{2}+s_{1}^{2}+\cdots+s_{b}^{2}}e^{2i\left(r_{1}x_{1}+\cdots+r_{a}x_{a}+s_{1}y_{1}+\cdots+s_{b}y_{b}\right)}.
\end{align}
\end{thm}

\begin{thm}\label{LUO5}
Let $a$ be even, $m$ and $n$ be any positive integers, $a$ and $b$ be any non-negative integers such that $a+b=n$, and $x_{1}, \ldots,x_{a},y_{1},\ldots, y_{b}$ be any complex numbers such that their sum be $0$. We have
\begin{align}\label{Luo50}
\sum_{k=0}^{mn-1}\prod_{j=1}^{a}\vartheta_1\left(z+x_{j}+\frac{k \pi
}{mn}\left|\tau\right.\right)\prod_{i=1}^{b}\vartheta_3\left(z+y_{i}+\frac{k\pi
}{mn}\left|\tau\right.\right)=R_{1,3}\left(m,n;\tau \right)\vartheta_3\left(mnz|{m^2n}\tau\right),
\end{align}
where
\begin{align}\label{Luo51}
R_{1,3}\left(m,n;\tau 
\right)=& mni^{a}q^{\frac{a}{4}}
\sum_{\substack{r_{1},\ldots,r_{a},s_{1},\ldots,s_{b}=-\infty
\\
2(r_{1}+\cdots+r_{a}+s_{1}+\cdots+s_{b})+a=0}}^{\infty}\left(-1\right)^{r_{1}+\cdots+r_{a}}\notag\\
& \times
q^{r_{1}^{2}+\cdots+r_{a}^{2}+s_{1}^{2}+\cdots+s_{b}^{2}+r_{1}+\cdots+r_{a}} e^{i\left(2r_{1}x_{1}+\cdots+2r_{a}x_{a}+2s_{1}y_{1}+\cdots+2s_{b}y_{b}+x_{1}+\cdots+x_{a}\right)}.
\end{align}
\end{thm}

\begin{thm}\label{LUO6}
Let $a$ be even, $m$ and $n$ be any positive integers, $a$ and $b$ be any non-negative integers such that $a+b=n$, and $x_{1}, \ldots,x_{a},y_{1},\ldots, y_{b}$ be any complex numbers such that their sum be $0$.
\begin{itemize}
\item When $mn$ is even, we have
\end{itemize}
\begin{align}\label{Luo60}
\sum_{k=0}^{mn-1}\prod_{j=1}^{a}\vartheta_1\left(z+x_{j}+\frac{k\pi
}{mn}\left|\tau\right.\right)\prod_{i=1}^{b}\vartheta_4\left(z+y_{i}+\frac{k\pi
}{mn}\left|\tau\right.\right)=R_{1,4}\left(m,n;\tau \right)\vartheta_3\left(mnz|{m^2n}\tau\right).
\end{align}
\begin{itemize}
\item When $mn$ is odd, we have
\end{itemize}
\begin{align}\label{Luo60a}
\sum_{k=0}^{mn-1}\prod_{j=1}^{a}\vartheta_1\left(z+x_{j}+\frac{k\pi
}{mn}\left|\tau\right.\right)\prod_{i=1}^{b}\vartheta_4\left(z+y_{i}+\frac{k\pi
}{mn}\left|\tau\right.\right)=R_{1,4}\left(m,n;\tau \right)\vartheta_4\left(mnz|{m^2n}\tau\right),
\end{align}
where
\begin{align}\label{Luo61}
R_{1,4}\left(m,n;\tau 
\right)=& mnq^{\frac{a}{4}}
\sum_{\substack{r_{1},\ldots,r_{a},s_{1},\ldots,s_{b}=-\infty
\\
2(r_{1}+\cdots+r_{a}+s_{1}+\cdots+s_{b})+a=0}}^{\infty} q^{r_{1}^{2}+\cdots+r_{a}^{2}+s_{1}^{2}+\cdots+s_{b}^{2}+r_{1}+\cdots+r_{a}} \notag\\
& \times
e^{i\left(2r_{1}x_{1}+\cdots+2r_{a}x_{a}+2s_{1}y_{1}+\cdots+2s_{b}y_{b}+x_{1}+\cdots+x_{a}\right)}.
\end{align}
\end{thm}

\begin{thm}\label{LUO7}
Let $a$ be even, $m$ and $n$ be any positive integers, $a$ and $b$ be any non-negative integers such that $a+b=n$, and $x_{1}, \ldots,x_{a},y_{1},\ldots, y_{b}$ be any complex numbers such that their sum be $0$. We have
\begin{align}\label{Luo70}
\sum_{k=0}^{mn-1}\prod_{j=1}^{a}\vartheta_2\left(z+x_{j}+\frac{k\pi
}{mn}\left|\tau\right.\right)\prod_{i=1}^{b}\vartheta_3\left(z+y_{i}+\frac{k\pi
}{mn}\left|\tau\right.\right)=R_{2,3}\left(m,n;\tau \right)\vartheta_3\left(mnz|{m^2n}\tau\right),
\end{align}
where
\begin{align}\label{Luo71}
R_{2,3}\left(m,n;\tau 
\right)=& mnq^{\frac{a}{4}}
\sum_{\substack{r_{1},\ldots,r_{a},s_{1},\ldots,s_{b}=-\infty
\\
2(r_{1}+\cdots+r_{a}+s_{1}+\cdots+s_{b})+a=0}}^{\infty}q^{r_{1}^{2}+\cdots+r_{a}^{2}+s_{1}^{2}+\cdots+s_{b}^{2}+r_{1}+\cdots+r_{a}}\notag\\
& \times
e^{i\left(2r_{1}x_{1}+\cdots+2r_{a}x_{a}+2s_{1}y_{1}+\cdots+2s_{b}y_{b}+x_{1}+\cdots+x_{a}\right)}.
\end{align}
\end{thm}

\begin{thm}\label{LUO8}
Let $a$ be even, $m$ and $n$ be any positive integers, $a$ and $b$ be any non-negative integers such that $a+b=n$, and $x_{1}, \ldots,x_{a},y_{1},\ldots, y_{b}$ be any complex numbers such that their sum be $0$.
\begin{itemize}
\item When $mb$ is even, we have
\end{itemize}
\begin{align}\label{Luo80}
\sum_{k=0}^{mn-1}\prod_{j=1}^{a}\vartheta_2\left(z+x_{j}+\frac{k\pi
}{mn}\left|\tau\right.\right)\prod_{i=1}^{b}\vartheta_4\left(z+y_{i}+\frac{k\pi
}{mn}\left|\tau\right.\right)=R_{2,4}\left(m,n;\tau \right)\vartheta_3\left(mnz|{m^2n}\tau\right).
\end{align}
\begin{itemize}
\item When $mb$ is odd, we have
\end{itemize}
\begin{align}\label{Luo80a}
\sum_{k=0}^{mn-1}\prod_{j=1}^{a}\vartheta_2\left(z+x_{j}+\frac{k\pi
}{mn}\left|\tau \right.\right)\prod_{i=1}^{b}\vartheta_4\left(z+y_{i}+\frac{k\pi
}{mn}\left|\tau\right.\right)=R_{2,4}\left(m,n;\tau \right)\vartheta_4\left(mnz|{m^2n}\tau\right),
\end{align}
where
\begin{align}\label{Luo81}
R_{2,4} \left(m,n;\tau \right)=& mnq^{\frac{a}{4}}
\sum_{\substack{r_{1},\ldots,r_{a},s_{1},\ldots,s_{b}=-\infty
\\
2(r_{1}+\cdots+r_{a}+s_{1}+\cdots+s_{b})+a=0}}^{\infty}\left(-1\right)^{s_{1}+\cdots+s_{b}}\notag\\
& \times
q^{r_{1}^{2}+\cdots+r_{a}^{2}+s_{1}^{2}+\cdots+s_{b}^{2}+r_{1}+\cdots+r_{a}}\notag\\
& \times
e^{i\left(2r_{1}x_{1}+\cdots+2r_{a}x_{a}+2s_{1}y_{1}+\cdots+2s_{b}y_{b}+x_{1}+\cdots+x_{a}\right)}.
\end{align}
\end{thm}

\begin{thm}\label{LUO9}
Let $m$ and $n$ be any positive integers, $a$ and $b$ be any non-negative integers such that $a+b=n$, and $x_{1}, \ldots,x_{a},y_{1},\ldots, y_{b}$ be any complex numbers such that their sum be $0$.
\begin{itemize}
\item When $mb$ is even, we have
\end{itemize}
\begin{align}\label{Luo90}
\sum_{k=0}^{mn-1}\prod_{j=1}^{a}\vartheta_3\left(z+x_{j}+\frac{k\pi
}{mn}\left|\tau\right.\right)\prod_{i=1}^{b}\vartheta_4\left(z+y_{i}+\frac{k\pi
}{mn}\left|\tau\right.\right)=R_{3,4}\left(m,n;\tau \right)\vartheta_3\left(mnz|{m^2n}\tau\right).
\end{align}
\begin{itemize}
\item When $mb$ is odd, we have
\end{itemize}
\begin{align}\label{Luo90a}
\sum_{k=0}^{mn-1}\prod_{j=1}^{a}\vartheta_3\left(z+x_{j}+\frac{k\pi
}{mn}\left|\tau\right.\right)\prod_{i=1}^{b}\vartheta_4\left(z+y_{i}+\frac{k\pi
}{mn}\left|\tau\right.\right)=R_{3,4}\left(\tau \right)\vartheta_4\left(mnz|{m^2n}\tau\right),
\end{align}
where
\begin{align}\label{Luo91}
R_{3,4} \left(m,n;\tau \right)=& mn
\sum_{\substack{r_{1},\ldots,r_{a},s_{1},\ldots,s_{b}=-\infty
\\
r_{1}+\cdots+r_{a}+s_{1}+\cdots+s_{b}=0}}^{\infty}(-1)^{s_{1}+\cdots+s_{b}}q^{r_{1}^{2}+\cdots+r_{a}^{2}+s_{1}^{2}+\cdots+s_{b}^{2}}\notag\\
& \times
e^{2i\left(r_{1}x_{1}+\cdots+r_{a}x_{a}+s_{1}y_{1}+\cdots+s_{b}y_{b}\right)}.
\end{align}
\end{thm}

\section{\bf Proofs of the main results}
In this section, we only prove \thmref{LUO4}, the proofs of other theorems are similar.

Let $f(z)$ be the left-hand of \eqref{Luo40} with $z \longmapsto \frac{z}{mn}, \tau \longmapsto \frac{\tau }{m^{2}n}$, we have
\begin{align}\label{Luo42}
f\left(z\right)=\sum_{k=0}^{mn-1}\prod_{j=1}^{a}\vartheta_1\left(\frac{z}{mn}+x_{j}+\frac{k\pi
}{mn}\left|\frac{\tau}{m^{2}n}\right.\right)\prod_{i=1}^{b}\vartheta_2\left(\frac{z}{mn}+y_{i}+\frac{k\pi
}{mn}\left|\frac{\tau}{m^{2}n}\right.\right).
\end{align}
By \eqref{Theta5} and \eqref{Theta6} and noting that $a+b=n$, we easily obtain
\begin{align}\label{Luo43}
f\left(z+\pi\right)=&\sum_{k=0}^{mn-1}\prod_{j=1}^{a}\vartheta_1\left(\frac{z+\pi}{mn}+x_{j}+\frac{k\pi
}{mn}\left|\frac{\tau}{m^{2}n}\right.\right)\prod_{i=1}^{b}\vartheta_2\left(\frac{z+\pi}{mn}+y_{i}+\frac{k\pi
}{mn}\left|\frac{\tau}{m^{2}n}\right.\right)\notag \\
=&\sum_{k=1}^{mn-1}\prod_{j=1}^{a}\vartheta_1\left(\frac{z}{mn}+x_{j}+\frac{k\pi
}{mn}\left|\frac{\tau}{m^{2}n}\right.\right)\prod_{i=1}^{b}\vartheta_2\left(\frac{z}{mn}+y_{i}+\frac{k\pi
}{mn}\left|\frac{\tau}{m^{2}n}\right.\right)\notag
\\&+(-1)^{n}\prod_{j=1}^{a}\vartheta_1\left(\frac{z}{mn}+x_{j}\left|\frac{\tau}{m^{2}n}\right.\right)\prod_{i=1}^{b}\vartheta_2\left(\frac{z}{mn}+y_{i}\left|\frac{\tau}{m^{2}n}\right.\right).
\end{align}
When $n$ is \textit{even}, comparing \eqref{Luo42} and \eqref{Luo43}, we get
\begin{align}\label{Luo44}
f(z)=f(z+\pi).
\end{align}
By \eqref{Theta9} and \eqref{Theta10}, noting that
$x_{1}+\cdots+x_{a}+y_{1}+\cdots+y_{b}=0$ and $a+b=n$, we have
\begin{align}\label{Luo45}
f\left(z+\pi\tau\right)=&\sum_{k=0}^{mn-1}\prod_{j=1}^{a}\vartheta_1\left(\frac{z+\pi\tau}{mn}+x_{j}+\frac{k\pi
}{mn}\left|\frac{\tau}{m^{2}n}\right.\right)\prod_{i=1}^{b}\vartheta_2\left(\frac{z+\pi\tau}{mn}+y_{i}+\frac{k\pi
}{mn}\left|\frac{\tau}{m^{2}n}\right.\right)\notag \\
            =&(-1)^{ma}q^{-1}e^{-2iz}\sum_{k=0}^{mn-1}\prod_{j=1}^{a}\vartheta_1\left(\frac{z}{mn}+x_{j}+\frac{k\pi
}{mn}\left|\frac{\tau}{m^{2}n}\right.\right)\prod_{i=1}^{b}\vartheta_2\left(\frac{z}{mn}+y_{i}+\frac{k\pi
}{mn}\left|\frac{\tau}{m^{2}n}\right.\right)\notag\\
            =&(-1)^{ma}q^{-1}e^{-2iz}f\left(z\right).
\end{align}
$\bullet $ When $ma$ is \textit{even} in \eqref{Luo45}, we have
\begin{align}\label{Luo45a}
f(z+\pi \tau)=q^{-1}e^{-2iz}f(z).
\end{align}
We construct the function
$\dfrac{f(z)}{\vartheta_3(z|\tau)}$, by \eqref{Theta7}, \eqref{Luo44} and \eqref{Luo45a},
we find that the function $\dfrac{f(z)}{\vartheta_3(z|\tau)}$ is an
elliptic function with the double periods $\pi$ and $\pi\tau$, and only
have a simple pole at $z=\dfrac{\pi}{2}+\dfrac{\pi\tau}{2}$ in the
period parallelogram. Hence the function
$\dfrac{f(z)}{\vartheta_3(z|\tau)}$ is a constant, say this constant is
${C}_{1,2}^{(1)}(m,n;\tau)$, i.e.,
\begin{equation*}
\frac{f(z)}{\vartheta_3(z|\tau)}={C}_{1,2}^{(1)}(m,n;\tau ),
\end{equation*}
we have
\begin{equation}\label{Luo46}
f(z)={C}_{1,2}^{(1)}(m,n;\tau )\vartheta_3(z|\tau).
\end{equation}
By \eqref{Luo42} and \eqref{Luo46}, we obtain
\begin{align}\label{Luo47}
\sum_{k=0}^{mn-1}\prod_{j=1}^{a}\vartheta_1\left(\frac{z}{mn}+x_{j}+\frac{k\pi
}{mn}\left|\frac{\tau}{m^{2}n}\right.\right)\prod_{i=1}^{b}\vartheta_2\left(\frac{z}{mn}+y_{i}+\frac{k\pi
}{mn}\left|\frac{\tau}{m^{2}n}\right.\right)={C}_{1,2}^{(1)}\left(m,n;\tau \right)\vartheta_3\left(z|\tau\right).
\end{align}
Letting $$z \longmapsto mnz \quad \textup{and} \quad \tau \longmapsto {m^{2}n} \tau$$ in \eqref{Luo47}, and then setting $$R_{1,2}^{(1)}\left(m,n;\tau \right)=C_{1,2}^{(1)}\left(m,n;{m^{2}n \tau }\right),$$ we arrive at \eqref{Luo40}.

Below we calculate the constant $R_{1,2}^{(1)}\left(m,n;\tau \right)$.

Setting $$z \longmapsto z+x_{j}+\frac{k \pi}{mn}$$ in
\eqref{Theta1}, by some simple calculation and noting that $n$ is even, we obtain
\begin{equation}\label{Luo198}
\begin{split}
\prod_{j=1}^{a}\vartheta_1\left(z+x_{j}+\frac{k \pi }{mn}\left|\tau \right.\right)=& i^{a}q^{\frac{a}{4}}e^{i(x_{1}+x_{2}+\cdots+x_{a})}\sum_{r_{1},\ldots,r_{a}=-\infty}^{\infty}(-1)^{r_{1}+\cdots+r_{a}}q^{r_{1}^{2}+\cdots+r_{a}^{2}+r_{1}+\cdots+r_{a}}\\
& \times e^{(2r_{1}+\cdots+2r_{a}+a)iz}
e^{2i (r_{1}x_{1}+\cdots+r_{a}x_{a})}e^{\frac{k\pi{i}}{mn}(2r_{1}+\cdots+2r_{a}+a)}.
\end{split}
\end{equation}

Setting $$z \longmapsto z+y_{i}+\frac{k \pi}{mn}$$ in
\eqref{Theta2}, we obtain
\begin{equation}\label{Luo198a}
\begin{split}
\prod_{i=1}^{b}\vartheta_2\left(z+y_{i}+\frac{k \pi }{mn}\left|\tau \right.\right)=& q^{\frac{b}{4}}e^{i(y_{1}+y_{2}+\cdots+y_{b})}\sum_{s_{1},\ldots,s_{b}=-\infty}^{\infty}q^{s_{1}^{2}+\cdots+s_{b}^{2}+s_{1}+\cdots+s_{b}}\\
& \times e^{(2s_{1}+\cdots+2s_{b}+b)iz}
e^{2i (s_{1}y_{1}+\cdots+s_{b}y_{b})}e^{\frac{k\pi{i}}{mn}(2s_{1}+\cdots+2s_{b}+b)}.
\end{split}
\end{equation}

Setting $$z \longmapsto mnz \quad \textup{and} \quad \tau \longmapsto {m^{2}n \tau}$$ in \eqref{Theta3}, we get
\begin{equation}\label{Luo199}
\vartheta_3\left(mnz|m^2n \tau \right)=\sum_{r=-\infty}^{\infty}q^{m^2nr^{2}}e^{2mnriz}.
\end{equation}
Substituting \eqref{Luo198}, \eqref{Luo198a} and \eqref{Luo199} into \eqref{Luo40}, noting that
$x_{1}+\cdots+x_{a}+y_{1}+\cdots+y_{b}=0$ and $a+b=n$, we find
\begin{align}\label{Luo48}
&i^{a}q^{\frac{n}{4}}\sum_{k=0}^{mn-1}\sum_{r_{1},\ldots,r_{a},s_{1},\ldots,s_{b}=-\infty}^{\infty}\left(-1\right)^{r_{1}+\cdots+r_{a}}q^{r_{1}^{2}+\cdots+r_{a}^{2}+s_{1}^{2}+\cdots+s_{b}^{2}+r_{1}+\cdots+r_{a}+s_{1}+\cdots+s_{b}}\notag \\
& \times
e^{iz\left(2r_{1}+\cdots+2r_{a}+2s_{1}+\cdots+2s_{b}+n\right)}e^{2i\left(r_{1}x_{1}+\cdots+r_{a}x_{a}+s_{1}y_{1}\cdots+s_{b}y_{b}\right)}e^{\frac{k\pi{i}
}{mn}\left(2r_{1}+\cdots+2r_{a}+2s_{1}+\cdots+2s_{b}+n\right)}\notag \\
& ={R}_{1,2}^{(1)}\left(m,n;\tau \right)\sum_{r=-\infty}^{\infty}q^{{m^{2}n} r^{2}}e^{2mnriz}.
\end{align}
Equating the constant of both sides of \eqref{Luo48}, noting that $r_{1}+\cdots+r_{a}+s_{1}+\cdots+s_{b}=-\frac{n}{2}$, we have

\begin{align}\label{Luo49}
{R}_{1,2}^{(1)}\left(m,n; \tau 
\right)=& mni^{a}q^{-\frac{n}{4}}\sum_{\substack{r_{1},\ldots,r_{a},s_{1},\ldots,s_{b}=-\infty
\\
2(r_{1}+\cdots+r_{a}+s_{1}+\cdots+s_{b})+n=0}}^{\infty}
(-1)^{r_{1}+\cdots+r_{a}}q^{r_{1}^{2}+\cdots+r_{a}^{2}+s_{1}^{2}+\cdots+s_{b}^{2}}\notag \\
& \times
 e^{2i \left
(r_{1}x_{1}+\cdots+r_{a}x_{a}+s_{1}y_{1}\cdots+s_{b}y_{b}\right)}.
\end{align}

$\bullet $ When $ma$ is \textit{odd} in \eqref{Luo45}, we have
\begin{align}\label{Luo45b}
f(z+\pi \tau)=-q^{-1}e^{-2iz}f(z).
\end{align}
We construct the function $\frac{f(z)}{\vartheta_4(z|\tau)}$, by \eqref{Theta8}, \eqref{Luo45a} and \eqref{Luo45b}, we find that the function $\frac{f(z)}{\vartheta_4(z|\tau)}$ is an elliptic function with double periods $\pi$ and $\pi\tau$, and has \textit{only} a simple pole at $z=\frac{\pi \tau}{2}$ in the period parallelogram. Hence the function $\frac{f(z)}{\vartheta_4(z|\tau)}$ is a constant, say $C_{1,2}^{(2)}(m,n;\tau)$,
 we have
\begin{equation*}
f(z)=C_{1,2}^{(2)}(m,n;\tau)\vartheta_4(z|\tau),
\end{equation*}
or, equivalently
\begin{align}\label{Luo615}
\sum_{k=0}^{mn-1}\prod_{j=1}^{a}\vartheta_1\left(\frac{z}{mn}+x_{j}+\frac{k\pi
}{mn}\left|\frac{\tau}{m^{2}n}\right.\right)\prod_{i=1}^{b}\vartheta_2\left(\frac{z}{mn}+y_{i}+\frac{k\pi
}{mn}\left|\frac{\tau}{m^{2}n}\right.\right)=C_{1,2}^{(2)}\left(m,n;\tau \right)\vartheta_{4}\left(z|\tau \right).
\end{align}
Letting $$z \longmapsto mnz \quad \textup{and} \quad \tau \longmapsto {m^{2}n} \tau$$ in \eqref{Luo615}, and then setting $$R_{1,2}^{(2)}\left(m,n;\tau \right)=C_{1,2}^{(2)}\left({m^{2}n \tau }\right),$$ we arrive at \eqref{Luo40a}.

Similarly, by \eqref{Theta1}, \eqref{Theta2} and \eqref{Theta4}, from \eqref{Luo40a}, we obtain
\begin{align}\label{Luo48a}
& i^{a}q^{\frac{n}{4}}\sum_{k=0}^{mn-1}\sum_{r_{1},\ldots,r_{a},s_{1},\ldots,s_{b}=-\infty}^{\infty}\left(-1\right)^{r_{1}+\cdots+r_{a}}q^{r_{1}^{2}+\cdots+r_{a}^{2}+s_{1}^{2}+\cdots+s_{b}^{2}+r_{1}+\cdots+r_{a}+s_{1}+\cdots+s_{b}}\notag \\
& \times
e^{iz\left(2r_{1}+\cdots+2r_{a}+2s_{1}+\cdots+2s_{b}+n\right)}e^{2i\left(r_{1}x_{1}+\cdots+r_{a}x_{a}+s_{1}y_{1}\cdots+s_{b}y_{b}\right)}e^{\frac{k\pi{i}
}{mn}\left(2r_{1}+\cdots+2r_{a}+2s_{1}+\cdots+2s_{b}+n\right)}\notag \\
& ={R}_{1,2}^{(2)}\left(m,n;\tau 
\right)\sum_{r=-\infty}^{\infty}(-1)^rq^{m^{2}nr^{2}}e^{2mnriz}.
\end{align}
Equating the constant of both sides of \eqref{Luo48a}, we have

\begin{align}\label{Luo49a}
{R}_{1,2}^{(2)}\left(m,n;\tau 
\right)=& mni^{a}q^{-\frac{n}{4}}\sum_{\substack{r_{1},\ldots,r_{a},s_{1},\ldots,s_{b}=-\infty
\\
2(r_{1}+\cdots+r_{a}+s_{1}+\cdots+s_{b})+n=0}}^{\infty}
\left(-1\right)^{r_{1}+\cdots+r_{a}}q^{r_{1}^{2}+\cdots+r_{a}^{2}+s_{1}^{2}+\cdots+s_{b}^{2}}\notag \\
& \times
 e^{2i\left
(r_{1}x_{1}+\cdots+r_{a}x_{a}+s_{1}y_{1}\cdots+s_{b}y_{b}\right)}.
\end{align}
Clearly, we have
$${R}_{1,2}\left(m,n;\tau 
\right):={R}_{1,2}^{(1)}\left(m,n;\tau 
\right)={R}_{1,2}^{(2)}\left(m,n;\tau 
\right).$$
The proof is complete.

\section{\bf Circular summation formulas for the products of the single Jacobi's theta functions}
In this section, we deduce the circular summation formulas for the products of the single Jacobi's theta functions from the above section.
\begin{thm}\label{cor149}
Suppose that $n$ is even, $m$ is any positive integers; $y_{1}, y_{2}, \ldots, y_{n}$ are any complex numbers such that $y_1+y_2+ \cdots+ y_n=0$, we have
\begin{align}\label{cor11}
\sum_{k=0}^{mn-1}\prod_{j=1}^{n}\vartheta_1\left(z+y_{j}+\frac{k \pi
}{mn}\left|\tau \right.\right)=R_{1}\left(m,n;\tau \right)\vartheta_3\left(mnz|m^2n \tau \right),
\end{align}
where
\begin{align}\label{cor12}
R_{1}\left(m,n;\tau \right)=mnq^{-\frac{n}{4}}
\sum_{\substack{r_{1}, \ldots, r_{n}=-\infty \\
r_{1}+\cdots+r_{n}=\frac{n}{2}}}^{\infty}q^{r^{2}_{1}+\cdots+r^{2}_{n}}e^{-2i\left(r_{1}y_{1}+\cdots+r_{n}y_{n}\right)}.
\end{align}
\end{thm}
\begin{proof}
Define that the empty product $\prod_{j=1}^{b}=1, \textup{when $b<1$}$. Setting $a=n$ and $b=0$ in \eqref{Luo40} of \thmref{LUO4}, noting that $r_{1}+\cdots+r_{n}=\frac{n}{2}$ and $n$ is even, we deduce \thmref{cor11}.
\end{proof}
\begin{thm}\label{cor155}
Suppose that $n$ is even, $m$ is any positive integers; $y_{1}, y_{2}, \ldots, y_{n}$ are any complex numbers such that $y_1+y_2+ \cdots+ y_n=0$, we have
\begin{align}\label{cor156}
\sum_{k=0}^{mn-1}\prod_{j=1}^{n}\vartheta_2\left(z+y_{j}+\frac{k \pi
}{mn}\left|\tau \right.\right)=R_{2}\left(m,n;\tau \right)\vartheta_3\left(mnz|m^2n \tau \right),
\end{align}
where
\begin{align}\label{cor157}
R_{2}\left(m,n;\tau \right)=mnq^{-\frac{n}{4}}
\sum_{\substack{r_{1}, \ldots, r_{n}=-\infty \\
r_{1}+\cdots+r_{n}=\frac{n}{2}}}^{\infty}q^{r^{2}_{1}+\cdots+r^{2}_{n}}e^{-2i\left(r_{1}y_{1}+\cdots+r_{n}y_{n}\right)}.
\end{align}
\end{thm}
\begin{proof}
Define that the empty product $\prod_{j=1}^{a}=1, \textup{when $a<1$}$. Setting $a=0$ and $b=n$ in \eqref{Luo40} of \thmref{LUO4}, noting that $s_{1}+\cdots+s_{n}=\frac{n}{2}$ and $n$ is even, we deduce \thmref{cor155}.
\end{proof}
\begin{thm}\label{cor600}
Suppose that $m,n$ are any positive integers; $y_{1}, y_{2}, \ldots, y_{n}$ are any complex numbers such that $y_1+y_2+ \cdots+ y_n=0$, we have
\begin{align}\label{cor601}
\sum_{k=0}^{mn-1}\prod_{j=1}^{n}\vartheta_3\left(z+y_{j}+\frac{k \pi 
}{mn}\left|\tau \right.\right)=R_{3}\left(m,n;\tau \right)\vartheta_3\left(mnz|m^2n \tau \right),
\end{align}
where
\begin{align}\label{cor602}
& R_{3}\left(m,n;\tau \right)=mn
 \sum_{\substack{r_{1}+\cdots+r_{n}=0\\ r_{1}, \ldots, r_{n}=-\infty}}^{\infty}q^{r^{2}_{1}+\cdots+r^{2}_{n}}e^{2i\left(r_{1}y_{1}+\cdots+r_{n}y_{n}\right)}.
\end{align}
\end{thm}
\begin{proof}
Define that the empty product $\prod_{j=1}^{b}=1, \textup{when $b<1$}$. Setting $a=n$ and $b=0$ in \eqref{Luo90} of \thmref{LUO9}, we deduce \thmref{cor600}.
\end{proof}
\begin{thm}\label{cor400}
Suppose that $mn$ is even; $y_{1}, y_{2}, \ldots, y_{n}$ are any complex numbers such that $y_1+y_2+ \cdots+ y_n=0$, we have
\begin{align}\label{cor401}
\sum_{k=0}^{mn-1}\prod_{j=1}^{n}\vartheta_4\left(z+y_{j}+\frac{k \pi
}{mn}\left|\tau \right.\right)=R_{4}\left(m,n;\tau \right)\vartheta_3\left(mnz|m^2n \tau \right),
\end{align}
where
\begin{align}\label{cor402}
R_{4}\left(m,n;\tau \right)=mn
 \sum_{\substack{
r_{1}+\cdots+r_{n}=0\\ r_{1}, \ldots, r_{n}=-\infty}}^{\infty}q^{r^{2}_{1}+\cdots+r^{2}_{n}}e^{2i\left(r_{1}y_{1}+\cdots+r_{n}y_{n}\right)}.
\end{align}
\end{thm}
\begin{proof}
Define that the empty product $\prod_{j=1}^{a}=1, \textup{when $a<1$}$. Setting $a=0$ and $b=n$ in \eqref{Luo90} of \thmref{LUO9}, we deduce \thmref{cor400}.
\end{proof}
\begin{rem}
\thmref{cor600} is just the main result of Chan and Liu (see \cite[p. 1191, Theorem 4]{ChanSH2010a}). \thmref{cor149}, \thmref{cor155} and \thmref{cor400} are some analogues of the main result (\thmref{Luo130}) of Chan and Liu.
\end{rem}

If taking $m=1$ and $y_1=y_2= \cdots= y_n=0$ in \thmref{cor149}, \thmref{cor155}, \thmref{cor600} and \thmref{cor400}, we further deduce the following interesting results.

\begin{cor}\label{cor200}
For $n$ is even, we have
\begin{align}\label{cor201}
\sum_{k=0}^{n-1}\vartheta_1^{n}\left(z+\frac{k \pi
}{n}\left|\tau \right.\right)=R_{1}\left(n;\tau \right)\vartheta_3\left(nz|n \tau \right),
\end{align}
where
\begin{align}\label{cor202}
R_{1}\left(n;\tau \right)=nq^{-\frac{n}{4}}
\sum_{\substack{r_{1}, \ldots, r_{n}=-\infty \\
r_{1}+\cdots+r_{n}=\frac{n}{2}}}^{\infty}q^{r^{2}_{1}+\cdots+r^{2}_{n}}.
\end{align}
\end{cor}

\begin{cor}\label{cor203}
For $n$ is even, we have
\begin{align}\label{cor204}
\sum_{k=0}^{n-1}\vartheta_2^{n}\left(z+\frac{k \pi
}{n}\left|\tau \right.\right)=R_{2}\left(n;\tau \right)\vartheta_3\left(nz|n \tau \right),
\end{align}
where
\begin{align}\label{cor205}
R_{2}\left(n;\tau \right)=nq^{-\frac{n}{4}}
\sum_{\substack{r_{1}, \ldots, r_{n}=-\infty \\
r_{1}+\cdots+r_{n}=\frac{n}{2}}}^{\infty}q^{r^{2}_{1}+\cdots+r^{2}_{n}}.
\end{align}
\end{cor}

\begin{cor}\label{cor206}
For $n$ is any positive integers, we have
\begin{align}\label{cor207}
\sum_{k=0}^{n-1}\vartheta_3^{n}\left(z+\frac{k \pi 
}{n}\left|\tau \right.\right)=R_{3}\left(n;\tau \right)\vartheta_3\left(nz|n \tau \right),
\end{align}
where
\begin{align}\label{cor208}
& R_{3}\left(n;\tau \right)=n
 \sum_{\substack{r_{1}+\cdots+r_{n}=0\\ r_{1}, \ldots, r_{n}=-\infty}}^{\infty}q^{r^{2}_{1}+\cdots+r^{2}_{n}}.
\end{align}
\end{cor}

\begin{cor}\label{cor209}
For $n$ is even, we have
\begin{align}\label{cor210}
\sum_{k=0}^{n-1}\vartheta_4^{n}\left(z+\frac{k \pi
}{n}\left|\tau \right.\right)=R_{4}\left(n;\tau \right)\vartheta_3\left(nz|n \tau \right),
\end{align}
where
\begin{align}\label{cor211}
R_{4}\left(n;\tau \right)=n
 \sum_{\substack{
r_{1}+\cdots+r_{n}=0\\ r_{1}, \ldots, r_{n}=-\infty}}^{\infty}q^{r^{2}_{1}+\cdots+r^{2}_{n}}.
\end{align}
\end{cor}
\begin{rem}
\thmref{cor206} is just Ramanujan's circular summation \thmref{Luo124}. \thmref{cor200}, \thmref{cor203} and \thmref{cor209} are some analogues of Ramanujan's circular summation.
\end{rem}

If setting $x_{1}=x_{2}=\cdots=x_{a}=x$ and $y_{1}=y_{2}=\cdots=y_{b}=y$ in \thmref{LUO4}--\thmref{LUO9} respectively, then we obtain the following interesting special cases.
\begin{thm}\label{LUO20}
Let $n$ be even, $m$ be any positive integers, $a$ and $b$ be non-negative integers such that $a+b=n$, and $x$ and $y$ be any complex numbers such that $ax+by=0$.
\begin{itemize}
\item When $ma$ is even, we have
\end{itemize}
\begin{align}\label{LUO21}
\sum_{k=0}^{mn-1}\vartheta_1^{a}\left(z+x+\frac{k\pi
}{mn}\left|\tau\right.\right)\vartheta_2^{b}\left(z+y+\frac{k\pi
}{mn}\left|\tau\right.\right)=R_{1,2}\left(x,y;m,n;\tau \right)\vartheta_3\left(mnz|{m^2n}\tau\right).
\end{align}
\begin{itemize}
\item When $ma$ is odd, we have
\end{itemize}
\begin{align}\label{LUO22}
\sum_{k=0}^{mn-1}\vartheta_1^{a}\left(z+x+\frac{k\pi
}{mn}\left|\tau\right.\right)\vartheta_2^{b}\left(z+y+\frac{k\pi
}{mn}\left|\tau\right.\right)=R_{1,2}\left(x,y;m,n;\tau \right)\vartheta_4\left(mnz|{m^2n}\tau\right),
\end{align}
where
\begin{align}\label{LUO23}
R_{1,2}\left(x,y;m,n;\tau \right)
=& mni^{a}q^{-\frac{n}{4}}
\sum_{\substack{r_{1},\ldots,r_{a},s_{1},\ldots,s_{b}=-\infty
\\
2(r_{1}+\cdots+r_{a}+s_{1}+\cdots+s_{b})+n=0}}^{\infty}\left(-1\right)^{r_{1}+\cdots+r_{a}}\notag \\
& \times
q^{r_{1}^{2}+\cdots+r_{a}^{2}+s_{1}^{2}+\cdots+s_{b}^{2}}e^{2i\left((r_{1}+\cdots+r_{a})x+(s_{1}+\cdots+s_{b})y \right)}.
\end{align}
\end{thm}

\begin{thm}\label{LUO24}
Let $a$ be even, $m$ and $n$ be any positive integers, $a$ and $b$ be non-negative integers such that $a+b=n$, and $x$ and $y$ be any complex numbers such that $ax+by=0$. We have
\begin{align}\label{LUO25}
\sum_{k=0}^{mn-1}\vartheta_1^{a}\left(z+x+\frac{k \pi
}{mn}\left|\tau\right.\right)\vartheta_3^{b}\left(z+y+\frac{k\pi
}{mn}\left|\tau\right.\right)=R_{1,3}\left(x,y;m,n;\tau \right)\vartheta_3\left(mnz|{m^2n}\tau\right),
\end{align}
where
\begin{align}\label{LUO26}
R_{1,3}\left(x,y;m,n;\tau 
\right)=& mni^{a}q^{\frac{a}{4}}
\sum_{\substack{r_{1},\ldots,r_{a},s_{1},\ldots,s_{b}=-\infty
\\
2(r_{1}+\cdots+r_{a}+s_{1}+\cdots+s_{b})+a=0}}^{\infty}\left(-1\right)^{r_{1}+\cdots+r_{a}}\notag\\
& \times
q^{r_{1}^{2}+\cdots+r_{a}^{2}+s_{1}^{2}+\cdots+s_{b}^{2}+r_{1}+\cdots+r_{a}} e^{i\left(2(r_{1}+\cdots+r_{a})x+2(s_{1}+\cdots+s_{b})y+ax \right)}.
\end{align}
\end{thm}

\begin{thm}\label{LUO27}
Let $a$ be even, $m$ and $n$ be any positive integers, $a$ and $b$ be non-negative integers such that $a+b=n$, and $x$ and $y$ be any complex numbers such that $ax+by=0$.
\begin{itemize}
\item When $mn$ is even, we have
\end{itemize}
\begin{align}\label{LUO28}
\sum_{k=0}^{mn-1}\vartheta_1^{a}\left(z+x+\frac{k\pi
}{mn}\left|\tau\right.\right)\vartheta_4^{b}\left(z+y+\frac{k\pi
}{mn}\left|\tau\right.\right)=R_{1,4}\left(x,y;m,n;\tau \right)\vartheta_3\left(mnz|{m^2n}\tau\right).
\end{align}
\begin{itemize}
\item When $mn$ is odd, we have
\end{itemize}
\begin{align}\label{LUO29}
\sum_{k=0}^{mn-1}\vartheta_1^{a}\left(z+x+\frac{k\pi
}{mn}\left|\tau\right.\right)\vartheta_4^{b}\left(z+y+\frac{k\pi
}{mn}\left|\tau\right.\right)=R_{1,4}\left(x,y;m,n;\tau \right)\vartheta_4\left(mnz|{m^2n}\tau\right),
\end{align}
where
\begin{align}\label{LUO30}
R_{1,4}\left(x,y;m,n;\tau 
\right)=& mnq^{\frac{a}{4}}
\sum_{\substack{r_{1},\ldots,r_{a},s_{1},\ldots,s_{b}=-\infty
\\
2(r_{1}+\cdots+r_{a}+s_{1}+\cdots+s_{b})+a=0}}^{\infty} q^{r_{1}^{2}+\cdots+r_{a}^{2}+s_{1}^{2}+\cdots+s_{b}^{2}+r_{1}+\cdots+r_{a}} \notag\\
& \times
e^{i\left(2(r_{1}+\cdots+r_{a})x+2(s_{1}+\cdots+s_{b})y+ax \right)}.
\end{align}
\end{thm}

\begin{thm}\label{LUO31}
Let $a$ be even, $m$ and $n$ be any positive integers, $a$ and $b$ be non-negative integers such that $a+b=n$, and $x$ and $y$ be any complex numbers such that $ax+by=0$. We have
\begin{align}\label{LUO32}
\sum_{k=0}^{mn-1}\vartheta_2^{a}\left(z+x+\frac{k\pi
}{mn}\left|\tau\right.\right)\vartheta_3^{b}\left(z+y+\frac{k\pi
}{mn}\left|\tau\right.\right)=R_{2,3}\left(x,y;m,n;\tau \right)\vartheta_3\left(mnz|{m^2n}\tau\right),
\end{align}
where
\begin{align}\label{LUO33}
R_{2,3}\left(x,y;m,n;\tau 
\right)=& mnq^{\frac{a}{4}}
\sum_{\substack{r_{1},\ldots,r_{a},s_{1},\ldots,s_{b}=-\infty
\\
2(r_{1}+\cdots+r_{a}+s_{1}+\cdots+s_{b})+a=0}}^{\infty}q^{r_{1}^{2}+\cdots+r_{a}^{2}+s_{1}^{2}+\cdots+s_{b}^{2}+r_{1}+\cdots+r_{a}}\notag\\
& \times
e^{i\left(2(r_{1}+\cdots+r_{a})x+2(s_{1}+\cdots+s_{b})y+ax \right)}.
\end{align}
\end{thm}

\begin{thm}\label{LUO34}
Let $a$ be even, $m$ and $n$ be any positive integers, $a$ and $b$ be non-negative integers such that $a+b=n$, and $x$ and $y$ be any complex numbers such that $ax+by=0$.
\begin{itemize}
\item When $mb$ is even, we have
\end{itemize}
\begin{align}\label{LUO35}
\sum_{k=0}^{mn-1}\vartheta_2^{a}\left(z+x+\frac{k\pi
}{mn}\left|\tau\right.\right)\vartheta_4^{b}\left(z+y+\frac{k\pi
}{mn}\left|\tau\right.\right)=R_{2,4}\left(x,y;m,n;\tau \right)\vartheta_3\left(mnz|{m^2n}\tau\right).
\end{align}
\begin{itemize}
\item When $mb$ is odd, we have
\end{itemize}
\begin{align}\label{LUO36}
\sum_{k=0}^{mn-1}\vartheta_2^{a}\left(z+x+\frac{k\pi
}{mn}\left|\tau \right.\right)\vartheta_4^{b}\left(z+y+\frac{k\pi
}{mn}\left|\tau\right.\right)=R_{2,4}\left(x,y;m,n;\tau \right)\vartheta_4\left(mnz|{m^2n}\tau\right),
\end{align}
where
\begin{align}\label{LUO37}
R_{2,4} \left(x,y;m,n;\tau \right)=& mnq^{\frac{a}{4}}
\sum_{\substack{r_{1},\ldots,r_{a},s_{1},\ldots,s_{b}=-\infty
\\
2(r_{1}+\cdots+r_{a}+s_{1}+\cdots+s_{b})+a=0}}^{\infty}\left(-1\right)^{s_{1}+\cdots+s_{b}}\notag\\
& \times
q^{r_{1}^{2}+\cdots+r_{a}^{2}+s_{1}^{2}+\cdots+s_{b}^{2}+r_{1}+\cdots+r_{a}}\notag\\
& \times
e^{i\left(2(r_{1}+\cdots+r_{a})x+2(s_{1}+\cdots+s_{b})y+ax \right)}.
\end{align}
\end{thm}

\begin{thm}\label{LUO38}
Let $m$ and $n$ be any positive integers, $a$ and $b$ be non-negative integers such that $a+b=n$, and $x$ and $y$ be any complex numbers such that $ax+by=0$.
\begin{itemize}
\item When $mb$ is even, we have
\end{itemize}
\begin{align}\label{LUO39}
\sum_{k=0}^{mn-1}\vartheta_3^{a}\left(z+x+\frac{k\pi
}{mn}\left|\tau\right.\right)\vartheta_4^{b}\left(z+y+\frac{k\pi
}{mn}\left|\tau\right.\right)=R_{3,4}\left(x,y;m,n;\tau \right)\vartheta_3\left(mnz|{m^2n}\tau\right).
\end{align}
\begin{itemize}
\item When $mb$ is odd, we have
\end{itemize}
\begin{align}\label{LUO40}
\sum_{k=0}^{mn-1}\vartheta_3^{a}\left(z+x+\frac{k\pi
}{mn}\left|\tau\right.\right)\vartheta_4^{b}\left(z+y+\frac{k\pi
}{mn}\left|\tau\right.\right)=R_{3,4}\left(\tau \right)\vartheta_4\left(mnz|{m^2n}\tau\right),
\end{align}
where
\begin{align}\label{LUO41}
R_{3,4} \left(x,y;m,n;\tau \right)=& mn
\sum_{\substack{r_{1},\ldots,r_{a},s_{1},\ldots,s_{b}=-\infty
\\
r_{1}+\cdots+r_{a}+s_{1}+\cdots+s_{b}=0}}^{\infty}(-1)^{s_{1}+\cdots+s_{b}}q^{r_{1}^{2}+\cdots+r_{a}^{2}+s_{1}^{2}+\cdots+s_{b}^{2}}\notag \\
& \times
e^{2i\left((r_{1}+\cdots+r_{a})x+(s_{1}+\cdots+s_{b})y \right)}.
\end{align}
\end{thm}

\begin{rem}
\thmref{LUO20}--\thmref{LUO38} are some analogues of \thmref{Luo128} of Zeng.
\end{rem}

\section{\bf Some identities of Jacobi's theta functions}
In this section, we obtain some interesting identities of Jacobi's theta functions from \thmref{LUO4}-\thmref{LUO9}.

The multiple theta series $a(y_{1},y_{2}|\tau), b(y_{1},y_{2}|\tau), c(y_{1},y_{2}|\tau), d(y_{1},y_{2}|\tau)$ are respectively defined by
\begin{align}\label{Theta80}
& a(y_{1},y_{2}|\tau)=\sum_{r_{1},r_{2}=-\infty}^{\infty}q^{r_{1}^{2}+r_{1}r_{2}+r_{2}^{2}}e^{2i\left(r_{1}y_{1}+r_{2}y_{2}\right)},\\
\label{Theta90}
& b(y_{1},y_{2}|\tau)=\sum_{r_{1},r_{2}=-\infty}^{\infty}(-1)^{r_{1}+r_{2}}q^{r_{1}^{2}+r_{1}r_{2}+r_{2}^{2}} e^{2i\left(r_{1}y_{1}+r_{2}y_{2}\right)},\\
\label{Theta93}
& c(y_{1},y_{2}|\tau)=\sum_{r_{1},r_{2}=-\infty}^{\infty} q^{r_{1}^{2}+r_{1}r_{2}+r_{2}^{2}+2r_{1}+2r_{2}}e^{i\left(2r_{1}y_{1}+2r_{2}y_{2}+y_{1}+y_{2}\right)},\\
\label{Theta92}
& d(y_{1},y_{2}|\tau)=\sum_{r_{1},r_{2}=-\infty}^{\infty}(-1)^{r_{1}+r_{2}}q^{r_{1}^{2}+r_{1}r_{2}+r_{2}^{2}+2r_{1}+2r_{2}}e^{i\left(2r_{1}y_{1}+2r_{2}y_{2}+y_{1}+y_{2}\right)}.
\end{align}
The corresponding cubic theta functions $a(\tau),b(\tau)$ and $c(\tau)$ are respectively defined by (see \cite{Borwein1991})
\begin{align}\label{Theta94}
& a(\tau)=a(0,0|\tau)=\sum_{r_{1}, r_{2}=-\infty}^{\infty}q^{r_{1}^{2}+r_{1}r_{2}+r_{2}^{2}},\\
\label{Theta95}
& b(\tau)=b(0,0|\tau)=\sum_{r_{1}, r_{2}=-\infty}^{\infty}(-1)^{r_{1}+r_{2}}q^{r_{1}^{2}+r_{1}r_{2}+r_{2}^{2}},\\
\label{Theta102}
& c(\tau)=c(0,0|\tau)=\sum_{r_{1},r_{2}=-\infty}^{\infty}q^{r_{1}^{2}+r_{1}r_{2}+r_{2}^{2}+2r_{1}+2r_{2}},\\
\label{Theta97}
& d(\tau)=d(0,0|\tau)=\sum_{r_{1},r_{2}=-\infty}^{\infty}(-1)^{r_{1}+r_{2}}q^{r_{1}^{2}+r_{1}r_{2}+r_{2}^{2}+2r_{1}+2r_{2}}.
\end{align}
\begin{itemize}
\item
Applications of \thmref{LUO4}.

Taking $m=n=2, a=b=1$ in \eqref{Luo40} of \thmref{LUO4}, then $x_{1}+y_{1}=0$.  Setting $x=x_{1}$, by \eqref{Luo41}, via some direct computations, we have
\begin{align*}
R_{1,2}\left(2,2;\tau \right)=-4q^{\frac14}\vartheta_{1}(2x|2\tau). 
\end{align*}
Hence we obtain the folllowing identity for products of two theta functions:
\begin{multline}
\vartheta_1\left(z+x\left|\tau\right.\right)\vartheta_2\left(z-x\left|\tau\right.\right)
-\vartheta_1\left(z-x\left|\tau\right.\right)\vartheta_2\left(z+x\left|\tau\right.\right)\\
+\vartheta_1\left(z+x+\frac{\pi}{4}\left|\tau\right.\right)\vartheta_2\left(z-x+\frac{\pi
}{4}\left|\tau\right.\right)
+\vartheta_1\left(z+x-\frac{\pi}{4}\left|\tau\right.\right)\vartheta_2\left(z-x-\frac{\pi
}{4}\left|\tau\right.\right)\\=-4q^{\frac14}\vartheta_{1}(2x|2\tau) \vartheta_{3}(4z|8\tau)
\end{multline}

Taking $n=2, m=a=b=1$ in \eqref{Luo40a} of \eqref{LUO4}, then $x_{1}+y_{1}=0$. Setting $x=x_{1}$, by \eqref{Luo41}, via some direct computations, we have
\begin{align}
R_{1,2}\left(1,2;\tau \right)=-2q^{\frac14}\vartheta_{1}(2x|2\tau).
\end{align}
Hence we obtain the folllowing identity for products of two theta functions:
\begin{align}
\vartheta_1\left(z+x\left|\tau\right.\right)\vartheta_2\left(z-x\left|\tau\right.\right)
-\vartheta_1\left(z-x\left|\tau\right.\right)\vartheta_2\left(z+x\left|\tau\right.\right)
=-2q^{\frac14}\vartheta_{1}(2x|2\tau) \vartheta_{4}(2z|2\tau).
\end{align}

\item
Applications of \thmref{LUO5}.

Taking $n=3, a=2,b=1$ in \eqref{Luo50} of \thmref{LUO5}, then $x_{1}+x_{2}+y_{1}=0$. By  \eqref{Luo51}, via some direct computations, we have
\begin{align*}
R_{1,3}\left(m,3;\tau \right)=-3m q^{\frac32} d(x_{1}-y_{1},x_{2}-y_{1}|2\tau).
\end{align*}
Hence we obtain the folllowing identity for products of three theta functions:
\begin{multline}\label{Theta98}
\sum_{k=0}^{3m-1}\vartheta_1\left(z+x_{1}+\frac{k \pi
}{3m}\left|\tau\right.\right)\vartheta_1\left(z+x_{2}+\frac{k \pi
}{3m}\left|\tau\right.\right) \vartheta_3\left(z+y_{1}+\frac{k\pi
}{3m}\left|\tau\right.\right)\\=-3m q^{\frac32} d(x_{1}-y_{1},x_{2}-y_{1}|2\tau)\vartheta_3\left(3mz|{3m^2}\tau\right).
\end{multline}

Further taking $m=1$ and $m=2$ in \eqref{Theta98} respectively, we get
\begin{multline}\label{Theta99}
\vartheta_1\left(z+x_{1}\left|\tau\right.\right)\vartheta_1\left(z+x_{2}\left|\tau\right.\right) \vartheta_3\left(z+y_{1}\left|\tau\right.\right)
+\vartheta_1\left(z+x_{1}+\frac{\pi}{3}\left|\tau\right.\right)\vartheta_1\left(z+x_{2}+\frac{\pi
}{3}\left|\tau\right.\right) \vartheta_3\left(z+y_{1}+\frac{\pi}{3}\left|\tau\right.\right)\\
+\vartheta_1\left(z+x_{1}-\frac{\pi}{3}\left|\tau\right.\right)\vartheta_1\left(z+x_{2}-\frac{\pi
}{3}\left|\tau\right.\right) \vartheta_3\left(z+y_{1}-\frac{\pi}{3}\left|\tau\right.\right)
=-3 q^{\frac32} d(x_{1}-y_{1},x_{2}-y_{1}|2\tau)\vartheta_3\left(3z|{3}\tau\right)
\end{multline}
and
\begin{multline}\label{Theta100}
\vartheta_2\left(z+x_{1}\left|\tau\right.\right)\vartheta_2\left(z+x_{2}\left|\tau\right.\right) \vartheta_4\left(z+y_{1}\left|\tau\right.\right)
+\vartheta_1\left(z+x_{1}\left|\tau\right.\right)\vartheta_1\left(z+x_{2}\left|\tau\right.\right) \vartheta_3\left(z+y_{1}\left|\tau\right.\right)\\
+\vartheta_1\left(z+x_{1}+\frac{\pi}{6}\left|\tau\right.\right)\vartheta_1\left(z+x_{2}+\frac{\pi
}{6}\left|\tau\right.\right)\vartheta_3\left(z+y_{1}+\frac{\pi}{6}\left|\tau\right.\right)\\
+\vartheta_1\left(z+x_{1}-\frac{\pi}{6}\left|\tau\right.\right)\vartheta_1\left(z+x_{2}-\frac{\pi
}{6}\left|\tau\right.\right) \vartheta_3\left(z+y_{1}-\frac{\pi}{6}\left|\tau\right.\right)\\
+\vartheta_1\left(z+x_{1}+\frac{\pi}{3}\left|\tau\right.\right)\vartheta_1\left(z+x_{2}+\frac{\pi
}{3}\left|\tau\right.\right) \vartheta_3\left(z+y_{1}+\frac{\pi}{3}\left|\tau\right.\right)\\
+\vartheta_1\left(z+x_{1}-\frac{\pi}{3}\left|\tau\right.\right)\vartheta_1\left(z+x_{2}-\frac{\pi
}{3}\left|\tau\right.\right) \vartheta_3\left(z+y_{1}-\frac{\pi}{3}\left|\tau\right.\right)
\\=-6 q^{\frac32} d(x_{1}-y_{1},x_{2}-y_{1}|2\tau)\vartheta_3\left(6z|{12}\tau\right).
\end{multline}
Setting $y_{1}=y_{2}=y_{3}=0$ in \eqref{Theta99} and \eqref{Theta100} respectively, we get 
\begin{multline}
\vartheta_1^2\left(z\left|\tau\right.\right) \vartheta_3\left(z\left|\tau\right.\right)
+\vartheta_1^2\left(z+\frac{\pi}{3}\left|\tau\right.\right)\vartheta_3\left(z+\frac{\pi}{3}\left|\tau\right.\right)+\vartheta_1^2\left(z-\frac{\pi}{3}\left|\tau\right.\right) \vartheta_3\left(z-\frac{\pi}{3}\left|\tau\right.\right)
=-3 q^{\frac32} d(2\tau)\vartheta_3\left(3z|{3}\tau\right).
\end{multline}
and
\begin{multline}
\vartheta_2^2\left(z+\left|\tau\right.\right)\vartheta_4\left(z\left|\tau\right.\right)
+\vartheta_1^2\left(z\left|\tau\right.\right) \vartheta_3\left(z\left|\tau\right.\right)\\
+\vartheta_1^2\left(z+\frac{\pi}{6}\left|\tau\right.\right)\vartheta_3\left(z+\frac{\pi}{6}\left|\tau\right.\right)
+\vartheta_1^2\left(z-\frac{\pi}{6}\left|\tau\right.\right)\vartheta_3\left(z-\frac{\pi}{6}\left|\tau\right.\right)\\
+\vartheta_1^2\left(z+\frac{\pi}{3}\left|\tau\right.\right)\vartheta_3\left(z+\frac{\pi}{3}\left|\tau\right.\right)
+\vartheta_1^2\left(z-\frac{\pi}{3}\left|\tau\right.\right)\vartheta_3\left(z-\frac{\pi}{3}\left|\tau\right.\right)\\
=-6 q^{\frac32} d(2\tau)\vartheta_3\left(6z|{12}\tau\right).
\end{multline}
The above $d(\tau)$ and $d(y_{1},y_{2}|\tau)$ are defined by \eqref{Theta92} and \eqref{Theta97} respectively.

\item
Applications of \thmref{LUO6}.

Taking $n=3, a=2,b=1$ in \eqref{Luo60} of \thmref{LUO6}, then $x_{1}+x_{2}+y_{1}=0$. By \eqref{Luo61}, via some direct computations, we have
\begin{align*}
R_{1,4}\left(m,3;\tau \right)=3m q^{\frac32} c(x_{1}-y_{1},x_{2}-y_{1}|2\tau).
\end{align*}
Hence we obtain the folllowing identity for products of three theta functions:
\begin{multline}\label{Theta101}
\sum_{k=0}^{3m-1}\vartheta_1\left(z+x_{1}+\frac{k \pi
}{3m}\left|\tau\right.\right)\vartheta_1\left(z+x_{2}+\frac{k \pi
}{3m}\left|\tau\right.\right) \vartheta_4\left(z+y_{1}+\frac{k\pi
}{3m}\left|\tau\right.\right)\\=\left \{
\begin{array}{ll}
3m q^{\frac32} c(x_{1}-y_{1},x_{2}-y_{1}|2\tau)\vartheta_3\left(3mz|{3m^2}\tau\right), & \quad  m \ is \ even,\\
3m q^{\frac32} c(x_{1}-y_{1},x_{2}-y_{1}|2\tau)\vartheta_4\left(3mz|{3m^2}\tau\right), & \quad  m \ is \ odd.
\end{array}\right.
\end{multline}
Taking $m=2$ in the first equation of \eqref{Theta101}, we have
\begin{multline}\label{Theta103}
\vartheta_2\left(z+x_{1}\left|\tau\right.\right)\vartheta_2\left(z+x_{2}\left|\tau\right.\right) \vartheta_3\left(z+y_{1}\left|\tau\right.\right)
+\vartheta_1\left(z+x_{1}\left|\tau\right.\right)\vartheta_1\left(z+x_{2}\left|\tau\right.\right) \vartheta_4\left(z+y_{1}\left|\tau\right.\right)\\
+\vartheta_1\left(z+x_{1}+\frac{\pi}{6}\left|\tau\right.\right)\vartheta_1\left(z+x_{2}+\frac{\pi
}{6}\left|\tau\right.\right)\vartheta_4\left(z+y_{1}+\frac{\pi}{6}\left|\tau\right.\right)\\
+\vartheta_1\left(z+x_{1}-\frac{\pi}{6}\left|\tau\right.\right)\vartheta_1\left(z+x_{2}-\frac{\pi
}{6}\left|\tau\right.\right) \vartheta_4\left(z+y_{1}-\frac{\pi}{6}\left|\tau\right.\right)\\
+\vartheta_1\left(z+x_{1}+\frac{\pi}{3}\left|\tau\right.\right)\vartheta_1\left(z+x_{2}+\frac{\pi
}{3}\left|\tau\right.\right) \vartheta_4\left(z+y_{1}+\frac{\pi}{3}\left|\tau\right.\right)\\
+\vartheta_1\left(z+x_{1}-\frac{\pi}{3}\left|\tau\right.\right)\vartheta_1\left(z+x_{2}-\frac{\pi
}{3}\left|\tau\right.\right) \vartheta_4\left(z+y_{1}-\frac{\pi}{3}\left|\tau\right.\right)
\\=6 q^{\frac32} c(x_{1}-y_{1},x_{2}-y_{1}|2\tau)\vartheta_3\left(6z|{12}\tau\right).
\end{multline}
Setting $y_{1}=y_{2}=y_{3}=0$ in \eqref{Theta103}, we have
\begin{multline}
\vartheta_2^2\left(z+\left|\tau\right.\right)\vartheta_3\left(z\left|\tau\right.\right)
+\vartheta_1^2\left(z\left|\tau\right.\right) \vartheta_4\left(z\left|\tau\right.\right)\\
+\vartheta_1^2\left(z+\frac{\pi}{6}\left|\tau\right.\right)\vartheta_4\left(z+\frac{\pi}{6}\left|\tau\right.\right)
+\vartheta_1^2\left(z-\frac{\pi}{6}\left|\tau\right.\right)\vartheta_4\left(z-\frac{\pi}{6}\left|\tau\right.\right)\\
+\vartheta_1^2\left(z+\frac{\pi}{3}\left|\tau\right.\right)\vartheta_4\left(z+\frac{\pi}{3}\left|\tau\right.\right)
+\vartheta_1^2\left(z-\frac{\pi}{3}\left|\tau\right.\right)\vartheta_4\left(z-\frac{\pi}{3}\left|\tau\right.\right)\\
=6 q^{\frac32} c(2\tau)\vartheta_3\left(6z|{12}\tau\right).
\end{multline}
Taking $m=1$ in the second equation of \eqref{Theta101}, we have
\begin{multline}\label{Theta104}
\vartheta_1\left(z+x_{1}\left|\tau\right.\right)\vartheta_1\left(z+x_{2}\left|\tau\right.\right) \vartheta_4\left(z+y_{1}\left|\tau\right.\right)
+\vartheta_1\left(z+x_{1}+\frac{\pi}{3}\left|\tau\right.\right)\vartheta_1\left(z+x_{2}+\frac{\pi
}{3}\left|\tau\right.\right) \vartheta_4\left(z+y_{1}+\frac{\pi}{3}\left|\tau\right.\right)\\
+
\vartheta_1\left(z+x_{1}-\frac{\pi}{3}\left|\tau\right.\right)\vartheta_1\left(z+x_{2}-\frac{\pi
}{3}\left|\tau\right.\right) \vartheta_4\left(z+y_{1}-\frac{\pi}{3}\left|\tau\right.\right)
=3 q^{\frac32} c(x_{1}-y_{1},x_{2}-y_{1}|2\tau)\vartheta_4\left(3z|{3}\tau\right).
\end{multline}
The above $c(\tau)$ and $c(y_{1},y_{2}|\tau)$ are defined by \eqref{Theta93} and \eqref{Theta102} respectively.

\item
Applications of \thmref{LUO7}.

Taking $n=3, a=2,b=1$ in \eqref{Luo70} of \thmref{LUO7}, then $x_{1}+x_{2}+y_{1}=0$. By \eqref{Luo71}, via some direct computations, we have
\begin{align*}
R_{2,3}\left(m,3;\tau \right)=3m q^{\frac32} c(x_{1}-y_{1},x_{2}-y_{1}|2\tau).
\end{align*}
Hence we obtain the folllowing identity for products of three theta functions:
\begin{multline}\label{Theta105}
\sum_{k=0}^{3m-1}\vartheta_2\left(z+x_{1}+\frac{k \pi
}{3m}\left|\tau\right.\right)\vartheta_2\left(z+x_{2}+\frac{k \pi
}{3m}\left|\tau\right.\right) \vartheta_3\left(z+y_{1}+\frac{k\pi
}{3m}\left|\tau\right.\right)\\=3m q^{\frac32} c(x_{1}-y_{1},x_{2}-y_{1}|2\tau)\vartheta_3\left(3mz|{3m^2}\tau\right).
\end{multline}

Taking $m=1$ and $m=2$ in \eqref{Theta105} respectively, we get
\begin{multline}\label{Theta106}
\vartheta_2\left(z+x_{1}\left|\tau\right.\right)\vartheta_2\left(z+x_{2}\left|\tau\right.\right) \vartheta_3\left(z+y_{1}\left|\tau\right.\right)
+\vartheta_2\left(z+x_{1}+\frac{\pi}{3}\left|\tau\right.\right)\vartheta_2\left(z+x_{2}+\frac{\pi
}{3}\left|\tau\right.\right) \vartheta_3\left(z+y_{1}+\frac{\pi}{3}\left|\tau\right.\right)\\
+\vartheta_2\left(z+x_{1}-\frac{\pi}{3}\left|\tau\right.\right)\vartheta_2\left(z+x_{2}-\frac{\pi
}{3}\left|\tau\right.\right) \vartheta_3\left(z+y_{1}-\frac{\pi}{3}\left|\tau\right.\right)
=3 q^{\frac32} c(x_{1}-y_{1},x_{2}-y_{1}|2\tau)\vartheta_3\left(3z|{3}\tau\right)
\end{multline}
and
\begin{multline}\label{Theta107}
\vartheta_1\left(z+x_{1}\left|\tau\right.\right)\vartheta_1\left(z+x_{2}\left|\tau\right.\right) \vartheta_4\left(z+y_{1}\left|\tau\right.\right)
+\vartheta_2\left(z+x_{1}\left|\tau\right.\right)\vartheta_2\left(z+x_{2}\left|\tau\right.\right) \vartheta_3\left(z+y_{1}\left|\tau\right.\right)\\
+\vartheta_2\left(z+x_{1}+\frac{\pi}{6}\left|\tau\right.\right)\vartheta_2\left(z+x_{2}+\frac{\pi
}{6}\left|\tau\right.\right)\vartheta_3\left(z+y_{1}+\frac{\pi}{6}\left|\tau\right.\right)\\
+\vartheta_2\left(z+x_{1}-\frac{\pi}{6}\left|\tau\right.\right)\vartheta_2\left(z+x_{2}-\frac{\pi
}{6}\left|\tau\right.\right) \vartheta_3\left(z+y_{1}-\frac{\pi}{6}\left|\tau\right.\right)\\
+\vartheta_2\left(z+x_{1}+\frac{\pi}{3}\left|\tau\right.\right)\vartheta_2\left(z+x_{2}+\frac{\pi
}{3}\left|\tau\right.\right) \vartheta_3\left(z+y_{1}+\frac{\pi}{3}\left|\tau\right.\right)\\
+\vartheta_2\left(z+x_{1}-\frac{\pi}{3}\left|\tau\right.\right)\vartheta_2\left(z+x_{2}-\frac{\pi
}{3}\left|\tau\right.\right) \vartheta_3\left(z+y_{1}-\frac{\pi}{3}\left|\tau\right.\right)
\\=6 q^{\frac32} c(x_{1}-y_{1},x_{2}-y_{1}|2\tau)\vartheta_3\left(6z|{12}\tau\right).
\end{multline}
Setting $y_{1}=y_{2}=y_{3}=0$ in \eqref{Theta106} and \eqref{Theta107} respectively, we get 
\begin{multline}
\vartheta_2^2\left(z\left|\tau\right.\right) \vartheta_3\left(z\left|\tau\right.\right)
+\vartheta_2^2\left(z+\frac{\pi}{3}\left|\tau\right.\right)\vartheta_3\left(z+\frac{\pi}{3}\left|\tau\right.\right)+\vartheta_2^2\left(z-\frac{\pi}{3}\left|\tau\right.\right) \vartheta_3\left(z-\frac{\pi}{3}\left|\tau\right.\right)
=3 q^{\frac32} c(2\tau)\vartheta_3\left(3z|{3}\tau\right)
\end{multline}
and
\begin{multline}
\vartheta_1^2\left(z+\left|\tau\right.\right)\vartheta_4\left(z\left|\tau\right.\right)
+\vartheta_2^2\left(z\left|\tau\right.\right) \vartheta_3\left(z\left|\tau\right.\right)\\
+\vartheta_2^2\left(z+\frac{\pi}{6}\left|\tau\right.\right)\vartheta_3\left(z+\frac{\pi}{6}\left|\tau\right.\right)
+\vartheta_2^2\left(z-\frac{\pi}{6}\left|\tau\right.\right)\vartheta_3\left(z-\frac{\pi}{6}\left|\tau\right.\right)\\
+\vartheta_2^2\left(z+\frac{\pi}{3}\left|\tau\right.\right)\vartheta_3\left(z+\frac{\pi}{3}\left|\tau\right.\right)
+\vartheta_2^2\left(z-\frac{\pi}{3}\left|\tau\right.\right)\vartheta_3\left(z-\frac{\pi}{3}\left|\tau\right.\right)\\
=6 q^{\frac32} c(2\tau)\vartheta_3\left(6z|{12}\tau\right).
\end{multline}
The above $c(\tau)$ and $c(y_{1},y_{2}|\tau)$ are defined by \eqref{Theta92} and \eqref{Theta97} respectively.

\item
Applications of \thmref{LUO8}.

Taking $n=3, a=2,b=1$ in \thmref{LUO8}, then $x_{1}+x_{2}+y_{1}=0$. By \eqref{Luo81}, via some direct computations, we have
\begin{align*}
R_{2,4}\left(m,3;\tau \right)=-3m q^{\frac32} d(x_{1}-y_{1},x_{2}-y_{1}|2\tau).
\end{align*}
Hence we obtain the folllowing identity for products of three theta functions:
\begin{multline}\label{Theta108}
\sum_{k=0}^{3m-1}\vartheta_2\left(z+x_{1}+\frac{k \pi
}{3m}\left|\tau\right.\right)\vartheta_2\left(z+x_{2}+\frac{k \pi
}{3m}\left|\tau\right.\right) \vartheta_4\left(z+y_{1}+\frac{k\pi
}{3m}\left|\tau\right.\right)\\=\left \{
\begin{array}{ll}
-3m q^{\frac32} d(x_{1}-y_{1},x_{2}-y_{1}|2\tau)\vartheta_3\left(3mz|{3m^2}\tau\right), & \quad  m \ is \ even,\\
-3m q^{\frac32} d(x_{1}-y_{1},x_{2}-y_{1}|2\tau)\vartheta_4\left(3mz|{3m^2}\tau\right), & \quad  m \ is \ odd.
\end{array}\right.
\end{multline}
Taking $m=2$ in the first equation of \eqref{Theta108}, we get
\begin{multline}\label{Theta109}
\vartheta_1\left(z+x_{1}\left|\tau\right.\right)\vartheta_1\left(z+x_{2}\left|\tau\right.\right) \vartheta_3\left(z+y_{1}\left|\tau\right.\right)
+\vartheta_2\left(z+x_{1}\left|\tau\right.\right)\vartheta_2\left(z+x_{2}\left|\tau\right.\right) \vartheta_4\left(z+y_{1}\left|\tau\right.\right)\\
+\vartheta_2\left(z+x_{1}+\frac{\pi}{6}\left|\tau\right.\right)\vartheta_2\left(z+x_{2}+\frac{\pi
}{6}\left|\tau\right.\right)\vartheta_4\left(z+y_{1}+\frac{\pi}{6}\left|\tau\right.\right)\\
+\vartheta_2\left(z+x_{1}-\frac{\pi}{6}\left|\tau\right.\right)\vartheta_2\left(z+x_{2}-\frac{\pi
}{6}\left|\tau\right.\right) \vartheta_4\left(z+y_{1}-\frac{\pi}{6}\left|\tau\right.\right)\\
+\vartheta_2\left(z+x_{1}+\frac{\pi}{3}\left|\tau\right.\right)\vartheta_2\left(z+x_{2}+\frac{\pi
}{3}\left|\tau\right.\right) \vartheta_4\left(z+y_{1}+\frac{\pi}{3}\left|\tau\right.\right)\\
+\vartheta_2\left(z+x_{1}-\frac{\pi}{3}\left|\tau\right.\right)\vartheta_2\left(z+x_{2}-\frac{\pi
}{3}\left|\tau\right.\right) \vartheta_4\left(z+y_{1}-\frac{\pi}{3}\left|\tau\right.\right)\\
=-6 q^{\frac32} d(x_{1}-y_{1},x_{2}-y_{1}|2\tau)\vartheta_3\left(6z|{12}\tau\right).
\end{multline}
Setting $y_{1}=y_{2}=y_{3}=0$ in \eqref{Theta103}, we get 
\begin{multline}
\vartheta_1^2\left(z+\left|\tau\right.\right)\vartheta_3\left(z\left|\tau\right.\right)
+\vartheta_2^2\left(z\left|\tau\right.\right) \vartheta_4\left(z\left|\tau\right.\right)\\
+\vartheta_2^2\left(z+\frac{\pi}{6}\left|\tau\right.\right)\vartheta_4\left(z+\frac{\pi}{6}\left|\tau\right.\right)
+\vartheta_2^2\left(z-\frac{\pi}{6}\left|\tau\right.\right)\vartheta_4\left(z-\frac{\pi}{6}\left|\tau\right.\right)\\
+\vartheta_2^2\left(z+\frac{\pi}{3}\left|\tau\right.\right)\vartheta_4\left(z+\frac{\pi}{3}\left|\tau\right.\right)
+\vartheta_2^2\left(z-\frac{\pi}{3}\left|\tau\right.\right)\vartheta_4\left(z-\frac{\pi}{3}\left|\tau\right.\right)\\
=-6 q^{\frac32} d(2\tau)\vartheta_3\left(6z|{12}\tau\right).
\end{multline}
Taking $m=1$ in the second equation of \eqref{Theta108}, we get
\begin{multline}\label{Theta110}
\vartheta_2\left(z+x_{1}\left|\tau\right.\right)\vartheta_2\left(z+x_{2}\left|\tau\right.\right) \vartheta_4\left(z+y_{1}\left|\tau\right.\right)\\
+\vartheta_2\left(z+x_{1}+\frac{\pi}{3}\left|\tau\right.\right)\vartheta_2\left(z+x_{2}+\frac{\pi
}{3}\left|\tau\right.\right) \vartheta_4\left(z+y_{1}+\frac{\pi}{3}\left|\tau\right.\right)\\
+
\vartheta_2\left(z+x_{1}-\frac{\pi}{3}\left|\tau\right.\right)\vartheta_2\left(z+x_{2}-\frac{\pi
}{3}\left|\tau\right.\right) \vartheta_4\left(z+y_{1}-\frac{\pi}{3}\left|\tau\right.\right)\\
=-3 q^{\frac32} d(x_{1}-y_{1},x_{2}-y_{1}|2\tau)\vartheta_4\left(3z|{3}\tau\right).
\end{multline}
Setting $y_{1}=y_{2}=y_{3}=0$ in \eqref{Theta110}, we get 
\begin{multline}
\vartheta_2^2\left(z\left|\tau\right.\right)\vartheta_4\left(z\left|\tau\right.\right)
+\vartheta_2^2\left(z+\frac{\pi}{3}\left|\tau\right.\right)\vartheta_4\left(z+\frac{\pi}{3}\left|\tau\right.\right)\\
+\vartheta_2^2\left(z-\frac{\pi}{3}\left|\tau\right.\right)\vartheta_4\left(z-\frac{\pi}{3}\left|\tau\right.\right)
=-3 q^{\frac32} d(2\tau)\vartheta_4\left(3z|{3}\tau\right).
\end{multline}
The above $d(\tau)$ and $d(y_{1},y_{2}|\tau)$ are defined by \eqref{Theta92} and \eqref{Theta97} respectively.

\item
Applications of \thmref{LUO9}.

{\bf Cases 1.} \   Taking $n=3, a=2,b=1$ in \thmref{LUO9}, then $x_{1}+x_{2}+y_{1}=0$. By \eqref{Luo91}, via some direct computations, we have
\begin{align*}
R_{3,4}\left(m,3;\tau \right)=3m b(x_{1}-y_{1},x_{2}-y_{1}|2\tau).
\end{align*}
Hence we obtain the folllowing identity for products of three theta functions:
\begin{multline}\label{Theta111}
\sum_{k=0}^{3m-1}\vartheta_3\left(z+x_{1}+\frac{k \pi
}{3m}\left|\tau\right.\right)\vartheta_3\left(z+x_{2}+\frac{k \pi
}{3m}\left|\tau\right.\right) \vartheta_4\left(z+y_{1}+\frac{k\pi
}{3m}\left|\tau\right.\right)\\=\left \{
\begin{array}{ll}
3m b(x_{1}-y_{1},x_{2}-y_{1}|2\tau)\vartheta_3\left(3mz|{3m^2}\tau\right), & \quad  m \ is \ even,\\
3m b(x_{1}-y_{1},x_{2}-y_{1}|2\tau)\vartheta_4\left(3mz|{3m^2}\tau\right), & \quad  m \ is \ odd.
\end{array}\right.
\end{multline}
Taking $m=2$ in the first equation of \eqref{Theta111}, we get
\begin{multline}\label{Theta112}
\vartheta_4\left(z+x_{1}\left|\tau\right.\right)\vartheta_4\left(z+x_{2}\left|\tau\right.\right) \vartheta_3\left(z+y_{1}\left|\tau\right.\right)
+\vartheta_3\left(z+x_{1}\left|\tau\right.\right)\vartheta_3\left(z+x_{2}\left|\tau\right.\right) \vartheta_4\left(z+y_{1}\left|\tau\right.\right)\\
+\vartheta_3\left(z+x_{1}+\frac{\pi}{6}\left|\tau\right.\right)\vartheta_3\left(z+x_{2}+\frac{\pi
}{6}\left|\tau\right.\right)\vartheta_4\left(z+y_{1}+\frac{\pi}{6}\left|\tau\right.\right)\\
+\vartheta_3\left(z+x_{1}-\frac{\pi}{6}\left|\tau\right.\right)\vartheta_3\left(z+x_{2}-\frac{\pi
}{6}\left|\tau\right.\right) \vartheta_4\left(z+y_{1}-\frac{\pi}{6}\left|\tau\right.\right)\\
+\vartheta_3\left(z+x_{1}+\frac{\pi}{3}\left|\tau\right.\right)\vartheta_3\left(z+x_{2}+\frac{\pi
}{3}\left|\tau\right.\right) \vartheta_4\left(z+y_{1}+\frac{\pi}{3}\left|\tau\right.\right)\\
+\vartheta_3\left(z+x_{1}-\frac{\pi}{3}\left|\tau\right.\right)\vartheta_3\left(z+x_{2}-\frac{\pi
}{3}\left|\tau\right.\right) \vartheta_4\left(z+y_{1}-\frac{\pi}{3}\left|\tau\right.\right)
\\=6 b(x_{1}-y_{1},x_{2}-y_{1}|2\tau)\vartheta_3\left(6z|{12}\tau\right).
\end{multline}
Setting $y_{1}=y_{2}=y_{3}=0$ in \eqref{Theta112}, we get 
\begin{multline}
\vartheta_4^2\left(z+\left|\tau\right.\right)\vartheta_3\left(z\left|\tau\right.\right)
+\vartheta_3^2\left(z\left|\tau\right.\right) \vartheta_4\left(z\left|\tau\right.\right)\\
+\vartheta_3^2\left(z+\frac{\pi}{6}\left|\tau\right.\right)\vartheta_4\left(z+\frac{\pi}{6}\left|\tau\right.\right)
+\vartheta_3^2\left(z-\frac{\pi}{6}\left|\tau\right.\right)\vartheta_4\left(z-\frac{\pi}{6}\left|\tau\right.\right)\\
+\vartheta_3^2\left(z+\frac{\pi}{3}\left|\tau\right.\right)\vartheta_4\left(z+\frac{\pi}{3}\left|\tau\right.\right)
+\vartheta_3^2\left(z-\frac{\pi}{3}\left|\tau\right.\right)\vartheta_4\left(z-\frac{\pi}{3}\left|\tau\right.\right)\\
=6 b(2\tau)\vartheta_3\left(6z|{12}\tau\right).
\end{multline}
Taking $m=1$ in the second equation of \eqref{Theta112}, we get
\begin{multline}\label{Theta113}
\vartheta_3\left(z+x_{1}\left|\tau\right.\right)\vartheta_3\left(z+x_{2}\left|\tau\right.\right) \vartheta_4\left(z+y_{1}\left|\tau\right.\right)\\
+\vartheta_3\left(z+x_{1}+\frac{\pi}{3}\left|\tau\right.\right)\vartheta_3\left(z+x_{2}+\frac{\pi
}{3}\left|\tau\right.\right) \vartheta_4\left(z+y_{1}+\frac{\pi}{3}\left|\tau\right.\right)\\
+
\vartheta_3\left(z+x_{1}-\frac{\pi}{3}\left|\tau\right.\right)\vartheta_3\left(z+x_{2}-\frac{\pi
}{3}\left|\tau\right.\right) \vartheta_4\left(z+y_{1}-\frac{\pi}{3}\left|\tau\right.\right)\\
=3 b(x_{1}-y_{1},x_{2}-y_{1}|2\tau)\vartheta_4\left(3z|{3}\tau\right).
\end{multline}
Setting $y_{1}=y_{2}=y_{3}=0$ in \eqref{Theta113}, we get 
\begin{multline}
\vartheta_3^2\left(z\left|\tau\right.\right)\vartheta_4\left(z\left|\tau\right.\right)
+\vartheta_3^2\left(z+\frac{\pi}{3}\left|\tau\right.\right)\vartheta_4\left(z+\frac{\pi}{3}\left|\tau\right.\right)
+\vartheta_3^2\left(z-\frac{\pi}{3}\left|\tau\right.\right)\vartheta_4\left(z-\frac{\pi}{3}\left|\tau\right.\right)
=3 b(2\tau)\vartheta_4\left(3z|{3}\tau\right).
\end{multline}
The above $b(\tau)$ and $b(y_{1},y_{2}|\tau)$ are defined by \eqref{Theta90} and \eqref{Theta95} respectively.

{\bf Cases 2.} \  Taking $n=3, a=1,b=2$ in \thmref{LUO9}, then $x_{1}+x_{2}+y_{1}=0$. By  \eqref{Luo91}, via some direct computations, we have
\begin{align*}
R_{3,4}\left(m,3;\tau \right)=3m b(y_{1}-x_{1},y_{2}-x_{1}|2\tau).
\end{align*}
Noting that $mb$ is even, we obtain the folllowing identity for products of three theta functions:
\begin{multline}\label{Theta114}
\sum_{k=0}^{3m-1}\vartheta_3\left(z+x_{1}+\frac{k \pi
}{3m}\left|\tau\right.\right)\vartheta_4\left(z+y_{1}+\frac{k \pi
}{3m}\left|\tau\right.\right) \vartheta_4\left(z+y_{2}+\frac{k\pi
}{3m}\left|\tau\right.\right)\\=3m b(y_{1}-x_{1},y_{2}-x_{1}|2\tau)\vartheta_3\left(3mz|{3m^2}\tau\right).
\end{multline}
Taking $m=2$ in  \eqref{Theta114}, we get
\begin{multline}\label{Theta115}
\vartheta_4\left(z+x_{1}\left|\tau\right.\right)\vartheta_3\left(z+y_{1}\left|\tau\right.\right) \vartheta_3\left(z+y_{2}\left|\tau\right.\right)
+\vartheta_3\left(z+x_{1}\left|\tau\right.\right)\vartheta_4\left(z+y_{1}\left|\tau\right.\right) \vartheta_4\left(z+y_{2}\left|\tau\right.\right)\\
+\vartheta_3\left(z+x_{1}+\frac{\pi}{6}\left|\tau\right.\right)\vartheta_4\left(z+y_{1}+\frac{\pi
}{6}\left|\tau\right.\right) \vartheta_4\left(z+y_{2}+\frac{\pi}{6}\left|\tau\right.\right)\\
+\vartheta_3\left(z+x_{1}-\frac{\pi}{6}\left|\tau\right.\right)\vartheta_4\left(z+y_{1}-\frac{\pi
}{6}\left|\tau\right.\right) \vartheta_4\left(z+y_{2}-\frac{\pi}{6}\left|\tau\right.\right)\\
+\vartheta_3\left(z+x_{1}+\frac{\pi}{3}\left|\tau\right.\right)\vartheta_4\left(z+y_{1}+\frac{\pi
}{3}\left|\tau\right.\right) \vartheta_4\left(z+y_{2}+\frac{\pi}{3}\left|\tau\right.\right)\\
+\vartheta_3\left(z+x_{1}-\frac{\pi}{3}\left|\tau\right.\right)\vartheta_4\left(z+y_{1}-\frac{\pi
}{3}\left|\tau\right.\right) \vartheta_4\left(z+y_{2}-\frac{\pi}{3}\left|\tau\right.\right)\\
=6 b(y_{1}-x_{1},y_{2}-x_{1}|2\tau)\vartheta_3\left(6z|{12}\tau\right).
\end{multline}
Setting $x_{1}=y_{1}=y_{2}=0$ in \eqref{Theta115}, we get 
\begin{multline}
\vartheta_4\left(z\left|\tau\right.\right)\vartheta_3^2\left(z\left|\tau\right.\right) 
+\vartheta_3\left(z\left|\tau\right.\right)\vartheta_4^2\left(z\left|\tau\right.\right)\\
+\vartheta_3\left(z+\frac{\pi}{6}\left|\tau\right.\right)\vartheta_4^2\left(z+\frac{\pi
}{6}\left|\tau\right.\right)
+\vartheta_3\left(z-\frac{\pi}{6}\left|\tau\right.\right)\vartheta_4^2\left(z-\frac{\pi
}{6}\left|\tau\right.\right)\\
+\vartheta_3\left(z+\frac{\pi}{3}\left|\tau\right.\right)\vartheta_4^2\left(z+\frac{\pi
}{3}\left|\tau\right.\right)
+\vartheta_3\left(z-\frac{\pi}{3}\left|\tau\right.\right)\vartheta_4^2\left(z-\frac{\pi
}{3}\left|\tau\right.\right)\\
=6 b(2\tau)\vartheta_3\left(6z|{12}\tau\right).
\end{multline}
Taking $m=1$ in  \eqref{Theta114}, we get
\begin{multline}\label{Theta116}
\vartheta_3\left(z+x_{1}\left|\tau\right.\right)\vartheta_4\left(z+y_{1}\left|\tau\right.\right) \vartheta_4\left(z+y_{2}\left|\tau\right.\right)\\
+\vartheta_3\left(z+x_{1}+\frac{\pi}{3}\left|\tau\right.\right)\vartheta_4\left(z+y_{1}+\frac{\pi
}{3}\left|\tau\right.\right) \vartheta_4\left(z+y_{2}+\frac{\pi}{3}\left|\tau\right.\right)\\
+\vartheta_3\left(z+x_{1}-\frac{\pi}{3}\left|\tau\right.\right)\vartheta_4\left(z+y_{1}-\frac{\pi
}{3}\left|\tau\right.\right) \vartheta_4\left(z+y_{2}-\frac{\pi}{3}\left|\tau\right.\right)\\
=3 b(2\tau)\vartheta_3\left(3z|{3}\tau\right).
\end{multline}
Setting $x_{1}=y_{1}=y_{2}=0$ in \eqref{Theta116}, we get 
\begin{equation}
\vartheta_3\left(z\left|\tau\right.\right)\vartheta_4^2\left(z\left|\tau\right.\right) 
+\vartheta_3\left(z+\frac{\pi}{3}\left|\tau\right.\right)\vartheta_4^2\left(z+\frac{\pi
}{3}\left|\tau\right.\right)
+\vartheta_3\left(z-\frac{\pi}{3}\left|\tau\right.\right)\vartheta_4^2\left(z-\frac{\pi
}{3}\left|\tau\right.\right)
=3 b(2\tau)\vartheta_3\left(3z|{3}\tau\right).
\end{equation}
The above $b(\tau)$ and $b(y_{1},y_{2}|\tau)$ are defined by \eqref{Theta90} and \eqref{Theta95} respectively.

\item
Applications of \thmref{cor400}.

Taking $n=3$ in \eqref{cor401} of \thmref{cor400}, for $y_{1}+y_{2}+y_{3}=0$, by \eqref{cor402} we obtain
$$R_{4,4}\left(y_{1},y_{2},;\tau\right)=3ma(y_{1}-y_{3}, y_{2}-y_{3}|2\tau).$$
We hence have
\begin{align}\label{Theta81}
\sum_{k=0}^{3m-1}\vartheta_{4}(z+y_{1}+\frac{k\pi}{3m}|\tau)\vartheta_{4}(z+y_{2}+\frac{k\pi}{3m}|\tau)\vartheta_{4}(z+y_{3}+\frac{k\pi}{3m}|\tau)=3ma(y_{1}-y_{3}, y_{2}-y_{3}|2\tau)\vartheta_{3}(3mz|3m^{2}\tau).
\end{align}
Taking $m=2$ in \eqref{Theta81}, we get
\begin{multline}\label{Theta82}
\vartheta_{3}(z+y_{1}|\tau)\vartheta_{3}(z+y_{2}|\tau)\vartheta_{3}(z+y_{3}|\tau)+\vartheta_{4}(z+y_{1}|\tau)\vartheta_{4}(z+y_{2}|\tau)\vartheta_{4}(z+y_{3}|\tau)\\
+\vartheta_{4}(z+y_{1}+\frac{\pi}{6}|\tau)\vartheta_{4}(z+y_{2}+\frac{\pi}{6}|\tau)\vartheta_{4}(z+y_{3}+\frac{\pi}{6}|\tau)\\
+\vartheta_{4}(z+y_{1}-\frac{\pi}{6}|\tau)\vartheta_{4}(z+y_{2}-\frac{\pi}{6}|\tau)\vartheta_{4}(z+y_{3}-\frac{\pi}{6}|\tau) \\
+\vartheta_{4}(z+y_{1}+\frac{\pi}{3}|\tau)\vartheta_{4}(z+y_{2}+\frac{\pi}{3}|\tau)\vartheta_{4}(z+y_{3}+\frac{\pi}{3}|\tau)\\
+\vartheta_{4}(z+y_{1}-\frac{\pi}{3}|\tau)\vartheta_{4}(z+y_{2}-\frac{\pi}{3}|\tau)\vartheta_{4}(z+y_{3}-\frac{\pi}{3}|\tau) \\
=6a(y_{1}-y_{3}, y_{2}-y_{3}|2\tau)\vartheta_{3}(6z|12\tau).
\end{multline}
Setting $y_{1}=y_{2}=y_{3}=0$ in \eqref{Theta82}, we get 
\begin{align}\label{Theta83}
\vartheta_{3}^3(z|\tau)+\vartheta_{4}^3(z|\tau)
+\vartheta_{4}^3(z+\frac{\pi}{6}|\tau)+\vartheta_{4}^3(z-\frac{\pi}{6}|\tau)+\vartheta_{4}^3(z+\frac{\pi}{3}|\tau)+\vartheta_{4}^3(z-\frac{\pi}{3}|\tau)=6a(2\tau)\vartheta_{3}(6z|12\tau).
\end{align}
The above $a(\tau)$ and $a(y_{1},y_{2}|\tau)$ are defined by \eqref{Theta80} and \eqref{Theta94} respectively.
\end{itemize}

\begin{rem}
In \cite{Cao2011a}, Zhu Cao obtained some identities for products of Ramanujan's theta functions $f(a,b)$. No doubt many interesting identities of theta functions can be formulated from the  \thmref{cor149}--\thmref{cor400} and \thmref{LUO20}--\thmref{LUO38}, we here omit them.
\end{rem}

\end{document}